\documentclass[11pt,reqno,british,a4paper,twoside]{article}
\usepackage[utf8]{inputenc}
\usepackage[sc]{mathpazo}
\usepackage{babel,mathtools,amsmath,amssymb,url,paralist,xspace,mathrsfs}
\usepackage{booktabs,microtype,array,fixltx2e,fancyhdr,amscd,csquotes,tikz}

\usepackage{pb-diagram}
\usepackage[amsmath,thmmarks]{ntheorem}

\mathtoolsset{mathic=true}
\setdefaultenum{\upshape (i)}{}{}{}

\theoremstyle{plain}
\theoremseparator{.}
\newtheorem{theorem}{Theorem}[section]
\newtheorem{proposition}[theorem]{Proposition}
\newtheorem{lemma}[theorem]{Lemma}
\newtheorem{corollary}[theorem]{Corollary}

\theorembodyfont{\normalfont}
\newtheorem{definition}[theorem]{Definition}

\theoremheaderfont{\itshape}
\theoremsymbol{\begin{small}\ensuremath{\quad\triangle}\end{small}}
\newtheorem{remark}[theorem]{Remark}
\theoremsymbol{\begin{small}\ensuremath{\quad\diamondsuit}\end{small}}
\newtheorem{example}[theorem]{Example}

\theoremstyle{nonumberplain}
\theoremheaderfont{\itshape}
\theorembodyfont{\normalfont}
\theoremsymbol{\begin{small}\ensuremath{\quad\Box}\end{small}}
\newtheorem{proof}{Proof}

\qedsymbol{\begin{small}\ensuremath{\quad\Box}\end{small}}

\numberwithin{equation}{section}
\numberwithin{table}{section}


\newcommand{\Lie}[1]{\operatorname{\textsl{#1}}}
\newcommand{\lie}[1]{\operatorname{\mathfrak{#1}}}
\newcommand{\Lel}[1]{{\mathsf{#1}}}

\newcommand{\SL}{\Lie{SL}}

\newcommand{\SU}{\Lie{SU}}
\newcommand{\Un}{\Lie{U}}
\newcommand{\g}{\lie{g}}

\newcommand{\lt}{\lie{t}}

\newcommand{\Ld}{\mathcal L}

\newcommand{\Hodge}{{*}}

\newcommand{\bC}{{\mathbb C}}
\newcommand{\bR}{{\mathbb R}}

\newcommand{\Hor}{\mathcal H}
\newcommand{\Ver}{\mathcal V}

\newcommand{\vp}{\varphi}
\newcommand{\wt}{\widetilde}

\DeclareMathOperator{\ad}{ad}
\DeclareMathOperator{\CP}{\bC P}
\DeclareMathOperator{\id}{id}
\DeclareMathOperator{\re}{Re}
\DeclareMathOperator{\Ric}{Ric}
\DeclareMathOperator{\Hol}{Hol}

\newcommand{\hook}{{\lrcorner\,}}

\newcommand{\ST}{\textsc{st}\xspace}
\newcommand{\SST}{\textsc{sst}\xspace}
\newcommand{\KT}{\textsc{kt}\xspace}
\newcommand{\SKT}{\textsc{skt}\xspace}

\newcommand{\NB}{\nabla}
\newcommand{\LC}{\NB^{\textup{LC}}}

\DeclarePairedDelimiter{\Span}{\langle}{\rangle}
\DeclarePairedDelimiter{\abs}{\lvert}{\rvert}

\newcommand{\sslash}{\mathbin{{/}\mkern-6mu{/}\mkern-2mu}}

\DeclareFontFamily{U}{bigeuf}{}
\DeclareFontShape{U}{bigeuf}{m}{n}{<-6>s*[1.5]eufm5%
<6-8>s*[1.5]eufm7%
<8->s*[1.5]eufm10}{}
\DeclareSymbolFont{bigeufletters}{U}{bigeuf}{m}{n}
\DeclareMathSymbol{\sumcic}{\mathop}{bigeufletters}{`S}

\newcommand{\br}{\hspace{0pt}}
\newcommand{\bdash}{-\br}

\begin{document}

\begin{center}
  \LARGE\bfseries{The odd side of torsion geometry}
\end{center}
\begin{center}
  \Large Diego Conti and Thomas Bruun Madsen
\end{center}

\begin{abstract}
We introduce and study a notion of `Sasaki with torsion structure' (\ST) as an odd-dimensional analogue of K\"ahler with torsion geometry (\KT). These are normal almost contact metric manifolds that admit a unique compatible connection with \( 3 \)-form torsion. Any odd-dimensional compact Lie group is shown to admit such a structure; in this case the structure is left-invariant and has closed torsion form.

We illustrate the relation between  \ST structures and other generalizations of Sasaki geometry, and explain how some standard constructions in Sasaki geometry can be adapted to this setting. In particular, we relate the \ST structure to a \KT structure on the space of leaves, and show that both the cylinder and the cone over an \ST manifold are \KT, although only the cylinder behaves well with respect to closedness of the torsion form. Finally, we introduce a notion of `\( G \)-moment map'. We provide criteria based on equivariant cohomology  ensuring the existence of these maps, and then apply them as a tool for reducing \ST structures.   
\end{abstract}

\bigskip
\begin{small}\noindent
\end{small}

\bigskip

\begin{small}\noindent
\end{small}
\bigskip

\section{Introduction}
\label{sec:intro}
Traditionally, one of the main reasons to study odd-dimensional Riemannian geometry has been the quest for new solutions to Einstein's equations. In this way, important contributions have been made, e.g., via the study of Sasaki-Einstein manifolds (cf. \cite{Boyer-G:Sasak-book}). In modern physical theories one encounters various generalizations of Einstein's equations. For example, type II string theory with constant dilation involves a Riemannian manifold \( (N,g) \) together with a triple \( (\NB,c,\psi) \) consisting of a metric connection \( \NB \) with \( 3 \)-form torsion \( c \), and a spinor field \( \psi \). This triple of data is then subject to the following constraints (\cite{Strominger:superstrings,Friedrich-I:spinors}): 

\begin{equation}
\label{eq:field-type-II}
\Ric^{\NB}=0, \quad d^\Hodge c=0,\quad \NB\psi=0,\quad c\cdot\psi=0.
\end{equation} 

In an insightful paper \cite{Friedrich-I:spinors}, Friedrich and Ivanov paved the way for the study of torsion geometry in odd dimensions. In particular, they indicated \cite[Theorem 8.2]{Friedrich-I:spinors} the boundaries of almost contact metric manifolds relevant in torsion geometry. Amongst the interesting odd-dimensional geometries, most of the attention has so far been centred around Sasaki manifolds, recently extending studies to include quasi-Sasaki structures (see, for instance, \cite{Puhle:qSasak}). However, it is worth noting that physical theories with non-constant dilation admit solutions that are not quasi-Sasaki (see, e.g., \cite[Theorem 2.6]{Fernandez-al:heterotic}). 

By drawing an analogy to torsion geometry in even dimensions, our attention is directed to the study of normal almost contact metric manifolds that come equipped with a unique compatible connection with skew-symmetric torsion; for such manifolds the Reeb vector field must be Killing (see Definition \ref{def:ST} and Proposition \ref{prop:STcharac}). This particular combination of integrability (normality) and Riemannian (Killing) input distinguishes our class of odd-dimensional torsion geometries, which we dub `\ST manifolds', or `\SST' if the torsion form is closed.    

The first important source of genuine \ST manifolds is provided by odd-dimensio\bdash{}nal compact Lie groups. Theorem~\ref{thm:cmp-Lie-grp} shows any odd-dimensional compact Lie group admits a left-invariant \ST structure, and moreover the associated torsion \( 3 \)-form is closed and coclosed. In particular, this class of examples satisfies the first two conditions of \eqref{eq:field-type-II}. One may regard these manifolds as the odd-dimensional analogue of the well known \cite{Spindel-al} \SKT structures on even-dimensional compact Lie groups. 

Turning from examples to classifications, we show in Proposition~\ref{prop:sandwich} that an \ST manifold is locally `sandwiched' between \KT manifolds. Another connection between \ST and \KT geometry is provided in terms of a cone construction, Proposition~\ref{prop:ST-KT-cone}, imitating the one known from Sasaki geometry. By replacing the cone with a cylinder, as in Proposition \ref{prop:ST-KT-cyl}, we are able to relate \ST and \KT manifolds in a way that preserves the property of having closed torsion \( 3 \)-form. 

In the presence of symmetry, meaning a freely acting, structure preserving, compact Lie group \( G \), a useful way of constructing new examples of Sasaki manifolds is via reduction. The final part of the paper extends this tool to \ST structures. In Definition \ref{def:G-moment}, we introduce the notion of a `\(G\)-moment map' for \ST manifolds. Proposition \ref{prop:reduction} then shows that the zero level set of a \( G \)-moment map reduced modulo symmetries, \mbox{\( \mu^{-1}(0)\slash G \)}, is again an \ST manifold.  These studies turn out to provide additional motivation for the particular definition of \ST manifolds (see Remark~\ref{rem:STred}).  To complete the description, beside illustrating the construction with examples, we provide conditions ensuring the existence of a \( G \)-moment map. Most significantly, Theorem \ref{thm:reduction} asserts that for any \SST manifold endowed with a symmetry group such a map exists, provided the \( \partial\overline{\partial} \) lemma and the equivariant \( \
partial\overline{\partial} \) lemma are satisfied, and the torsion extends to a closed equivariant \( 3 \)-form.

Whilst completing the paper, another notion of \ST manifold was introduced independently in \cite{Houri-al:ST}. Fortunately, there is no conflict in terminology, since the two definitions are equivalent (see Remark~\ref{rem:STvsST}). Interestingly, this reference is partly motivated by physics. As a consequence, one obtains a first explicit application of \ST manifolds in the context of supergravity theories.

\section{Sasaki with torsion structures}
\label{sec:ST}

While Sasaki manifolds already come with a compatible connection of skew-sym\bdash{}metric torsion, a systematic study reveals that the appropriate odd-dimensional analogue of \KT geometry include a larger subclass of the class of normal almost contact metric manifolds. In this section, we introduce the notion of \ST manifolds, discuss the fundamental theory, and supplement by a variety of examples.

\subsection{Basic definitions}

In order to motivate our particular notion of `Sasaki geometry with torsion', we first recall that in the Hermitian setting, a \KT manifold is a Hermitian manifold \( (M,g,J) \) together with a Hermitian connection \( \NB \) that has skew-symmetric torsion, meaning that \( \NB g=0=\NB J \) and that the map \[ (X,Y,Z)\mapsto g(\NB_XY-\LC_XY,Z) \]  is a \( 3 \)-form. Such a connection always exists, and moreover it is unique. Indeed, one \cite{Gauduchon:Dirac} finds that \[ \NB=\LC+T\slash2, \] where \( g(T(X,Y),Z)= d\omega(JX,JY,JZ) \).   

The odd-dimensional replacement of \KT geometry takes as its starting point a \emph{normal almost contact metric manifold \( (N^n,g,\xi,\eta,\vp) \)}.  Thus \( (N,g) \) is an odd-dimensional Riemannian manifold equipped with a unit-norm vector field \( \xi \), its dual \( 1 \)-form \( \eta=\xi^\flat=g(\xi,\cdot) \), and a bundle endomorphism \( \vp\colon TN\to TN \) such that the following conditions hold:  
\begin{equation}
\begin{gathered}
\label{eq:acm-con}
\vp(\xi)=0,\, \vp^2=-1+\eta\otimes\xi,\\
g(\vp(X),\vp(Y))=g(X,Y)-\eta(X)\eta(Y).
\end{gathered}
\end{equation}

The normality refers to the vanishing of the Sasaki-Hatakeyama tensor (\cite{Sasaki-H:Nijenhuis})
\begin{equation}
\label{eq:Sasaki-Htensor}
S=[\vp(X),\vp(Y)]+\vp^2[X,Y]-\vp[\vp(X),Y]-\vp[X,\vp(Y)]+d\eta(X,Y)\xi,
\end{equation}
for all vector fields \( X,Y \) on \( N \).  The fundamental \( 2 \)-form of \( (g,\xi,\eta,\vp) \) is defined by \( F:=g(\vp\cdot,\cdot) \).

The object of study in this paper is the following odd-dimensional cousin of \KT geometry:

\begin{definition}
\label{def:ST}
A \emph{Sasaki with torsion manifold}, briefly an \ST manifold, is a normal almost contact metric manifold \( (N,g,\xi,\eta,\vp) \) such that \( \xi \) is a Killing vector field. 
\end{definition}

One may rewrite the Killing condition to get the following alternative characterisation of \ST manifolds:

\begin{proposition}
\label{prop:ST-charac-deriv}
A normal almost contact metric manifold \( (N,g,\xi,\eta) \) is \ST if and only if \( \Ld_\xi F= 0=\xi\hook dF \).
\end{proposition}
\begin{proof}
The equivalence follows from the definition and the relations
\[\xi\hook d\eta =0= \Ld_\xi\phi,\]
which always hold on an normal almost contact metric structure (see \cite[Lemma 2.1]{Blair:quasi-Sasaki}). 
\end{proof}

At a first glance the notion of \ST geometry may seem somewhat remote from that of \KT geometry. This gap is bridged by using observations from \cite{Friedrich-I:spinors}. By an \emph{almost contact metric connection} we shall mean a connection \( \NB \) that preserves the (normal) almost contact metric structure: \[ \NB g=0=\NB\xi=\NB\vp. \] 
Indicating by \( d^\vp \)  the operator \[ d^\vp \alpha=d\alpha(\vp\cdot,\dotsc,\vp\cdot), \]
we can associate to any almost contact metric manifold  the \( 3 \)-form \[ c=\eta\wedge d\eta+d^\vp F. \] 
The following result, based on \cite[Theorem 8.2]{Friedrich-I:spinors}, now relates our definition of \ST manifolds to the study of connections with skew-symmetric torsion.

\begin{proposition}
\label{prop:STcharac}
A normal almost contact metric manifold \( (N,g,\xi,\eta,\vp) \) is an \ST manifold if and only if the map \( \nabla\colon \Gamma(TN)\to \Gamma(TN)\otimes \Gamma(T^*N) \), defined via \[ g(\NB_XY,Z)=g(\LC_XY,Z)+c(X,Y,Z)\slash2, \] is an almost contact metric connection. 
\end{proposition}
\begin{proof}
Clearly, \( \NB \) defines a metric connection; it is the unique metric connection whose torsion satisfies the condition \[ g(\NB_XY-\NB_YX-[X,Y],Z)=c(X,Y,Z), \] for all vector fields \( X,Y,Z \). 

Assume now that \( \NB \) is an almost contact metric connection. In particular, we have \( \NB\xi=0 \). This implies \( \xi \) is a Killing field:
\begin{equation*}
\begin{split}
0&=2(g(\NB_X\xi,Y)+g(X,\NB_Y\xi))-(c(X,\xi,Y)+c(X,Y,\xi))\\
&=2(g(\LC_X\xi,Y)+g(X,\LC_Y\xi)).
\end{split}
\end{equation*} 
Consequently, \( (g,\xi,\eta,\vp) \) defines an \ST structure on \( N \).

Conversely, let us suppose \( (g,\xi,\eta,\vp) \) is \ST. To show that \( \NB \) is an almost contact metric connection, we have to verify the conditions \( \NB\xi=0=\NB\vp \). The first of these follows by noting that the Killing condition implies \( d\eta(X,Y)=2(\LC_X\eta)(Y) \) since
\begin{equation*}
(\LC_X\eta)(Y)=g(\LC_X\xi,Y)=-g(\LC_Y\xi,X)=-(\LC_Y\eta)(X).
\end{equation*}
Whence,
\begin{equation*}
\begin{split}
2g(\NB_X\xi,Y)&=2g(\LC_X\xi,Y)+c(X,\xi,Y)=2(\LC_X\eta)(Y) -c(\xi,X,Y)\\
     &=d\eta(X,Y)-d\eta(X,Y)=0,
\end{split}
\end{equation*}
and \( \NB\xi=0 \), as required. 

Finally, we are left to prove
\begin{equation}
\label{eqn:nablaphi0}
2g((\NB_X\vp)(Y),Z)=2g((\LC_X\vp)(Y),Z)+c(X,Y,\vp(Z)) + c(X,\vp(Y),Z)
\end{equation}
is zero. Since all terms are tensorial in \( X,Y,Z \), we can proceed via a case-by-case analysis. Firstly, we distinguish between the two possibilities: \( Z=\xi \) or \( Z\perp\xi \). In the former, the term \( c(X,Y,\vp(Z)) \) vanishes identically, and \eqref{eqn:nablaphi0} reduces to
\begin{multline*}
2g((\LC_X\vp)(Y),\xi)+ d\eta(X,\vp(Y))\\
=2g(\xi,\NB_X(\vp(Y)))-c(\xi,X,\vp(Y))+d\eta(X,\vp(Y))=0.
 \end{multline*}

For the remaining case, \( Z\perp\xi \), note that  by normality of an \ST structure, we may also write 
\[ 2g((\LC_X\vp)(Y),Z)=dF(X,Y,Z)-dF(X,\vp(Y),\vp(Z))+\eta(Y)d\eta(X,\vp(Z)), \] (cf. \cite[p. 25]{Friedrich-I:spinors}).

Now consider the \( 2 \) possibilities \( Y=\xi \) and \( Y\perp\xi \). In the first case, \eqref{eqn:nablaphi0} becomes  \[ d\eta(X,\vp(Z))+c(X,\xi,\vp(Z))=0. \] 

In the second case, first take \( X=\xi \); then 
 \eqref{eqn:nablaphi0} reduces to 
\begin{equation*}
\begin{split}
 d\eta(Y,\vp(Z))+d\eta(\vp(Y),Z),
\end{split}
\end{equation*}
which is zero because, by normality, \( d\eta \) is of type \( (1,1) \) with respect to \( \vp \) (see, e.g., \cite[p. 333]{Blair:quasi-Sasaki}).

Finally, for the case \( X,Y,Z\perp\xi \), \eqref{eqn:nablaphi0} becomes
\begin{equation*}
dF(X,Y,Z)-dF(X,\vp(Y),\vp(Z))-dF(\vp(X),\vp(Y),Z)-dF(\vp(X),Y,\vp(Z)),
\end{equation*}
which also vanishes, since normality implies that \( dF \) has type \( (2,1)+(1,2) \) with respect to \( \vp \) (cf. \cite[Proposition 8.1]{Friedrich-I:spinors}).
\end{proof}

\begin{remark}
\label{rem:ST-KT-comp}
In the light of Proposition \ref{prop:STcharac}, it should come as no surprise that \KT and \textsc{st} geometries share many common features. However, some differences are inevitable. Notable features are:
\begin{compactenum}
\item by \cite[Theorem  8.2]{Friedrich-I:spinors}, \( \NB \) is the \emph{unique} almost contact metric connection with \( 3 \)-form torsion, so we may unambiguously refer to it as `the \ST connection' of \( (N,g,\xi,\eta,\vp) \).  If the associated torsion \( 3 \)-form is closed, \( dc=0 \), we shall say the Sasaki structure with torsion is \emph{strong}, or briefly refer to it as an \textsc{sst} structure. 
\item Since \( \NB \) is an almost contact metric connection, its restricted holonomy group is contained in \( 1\times\Un(k) \), \( n=2k+1 \). 
\item as noted in the proof of Proposition \ref{prop:STcharac}, \( d\eta \) is of type \( (1,1) \) with respect to \( \vp \), and \( dF \) has type \( (2,1)+(1,2) \). Consequently, the torsion \( 3 \)-form is of type \( (2,1)+(1,2) \) with respect to \( \vp \): \[ c(X,Y,Z)=c(\vp(X),\vp(Y),Z)+c(\vp(X),Y,\vp(Z))+c(X,\vp(Y),\vp(Z)).\]
\item in contrast with the K\"ahler setting, Sasaki manifolds (\( 2F=d\eta \)) already come with a compatible connection that has non-vanishing skew-symmetric torsion, \( c=2\eta\wedge F\neq0 \). On a Sasaki manifold the `horizontal' component \( d^\vp F \) of the torsion  is clearly zero, but the \SST condition never holds in dimensions \( \geqslant5 \) since \( d\eta\wedge d\eta\neq0 \).
\end{compactenum}
\end{remark}

The class of \ST manifolds differs from the classically studied subclasses of normal almost contact metric manifolds (see, e.g., \cite{Chinea-G:acm-clsI,Chinea-M:acm-clsII,Bazzoni-O:coK}). However, by using Proposition \ref{prop:STcharac}, we can express certain relations in terms of the torsion \( 3 \)-form:

\begin{proposition}
\label{prop:ST-charac-c}
An \ST manifold \( (N,g,\xi,\eta,\vp) \) is:
\begin{compactenum}
\item \emph{cosymplectic} if and only if the torsion vanishes, \( c=0 \);
\item \emph{co-K\"ahler} if and only if \( c=0 \)
\item \emph{quasi-Sasaki} if and only if \( c=\eta\wedge d\eta \);
\item  \emph{\( \alpha \)-Sasaki} if and only if \( c=\alpha\eta\wedge F \) for some \( \alpha\in\bR\setminus\{0\} \); 

(Sasaki is the case \( \alpha=2 \)) 
\item  \emph{\( \alpha \)-Kenmotsu} if and only if \( c=0 \); 
\item \emph{trans-Sasaki} if and only if \( c=\alpha\eta\wedge F \) for some \( \alpha\in \bR \).
\end{compactenum}
\end{proposition}
\begin{proof}
The assertions are easily derived from definitions. A cosymplectic manifold has \( dF=0=d\eta \), so (i) is immediate. Likewise for (ii), since co-K\"ahler means cosymplectic and normal. (iii) follows by using that quasi-Sasaki manifolds have \( dF=0 \) and that \ST manifolds have \( \xi\hook dF=0 \). By definition, an \( \alpha \)-Sasaki manifold has \( d\eta = \alpha F \) with \( \alpha\in\bR\setminus\{0\} \), so (iv) is obvious. Concerning the statement about \( \alpha \)-Kenmotsu manifolds, recall that these are characterised by having \( d\eta=0 \), \( dF=\alpha \eta\wedge F \) and now combine this with the \ST condition \( \xi\hook dF=0 \). Finally, the assertion in (vi) follows since any trans-Sasaki manifold, for which \( \xi \) is a Killing vector field, is either \( \alpha \)-Sasaki or cosymplectic (\cite{Marrero:trans-Sasak}).  
\end{proof}

\begin{example}
Consider the connected Lie group \( G \) (from \cite[Example 3]{Fernandez-al:HKT-amc}) which in terms of a basis \( \Lel e^1,\Lel e^2,\Lel e^3,\Lel e^4 \) of \( \g^* \) is determined via the structural relations \[ d\Lel e^1=0=d\Lel e^2,\, d\Lel e^3=\Lel e^3\wedge \Lel e^1 +\Lel e^4\wedge \Lel e^2,\, d\Lel e^4=\Lel e^4\wedge \Lel e^1 +\Lel e^2\wedge \Lel e^3. \] On the direct product \( N=\bR\times G \), we define an almost contact metric structure such that \( \Lel e^1,\dotsc, \Lel e^4,dt \) corresponds to an orthonormal coframe, and \[\Lel e^1\circ\vp=\Lel e^2,\,\Lel e^3\circ\vp=e^4, dt\circ\vp=0. \] Then, the resulting \ST structure has non-zero torsion \( 3 \)-form which is proportional to \( \Lel e^2\wedge\Lel e^3\wedge \Lel e^4 \). In particular, we see that \( dc\neq0 \) so that prescribed \ST structure on \( N \) is not \SST, although the `vertical' component \( \eta\wedge d\eta \) of the torsion is zero.
\end{example}

\begin{example}
\label{ex:quasi-Sasak-SKT}
In \cite[Section 2.1]{Fernandez-al:SKT}, one finds a number of examples of quasi-Sasaki Lie algebras satisfying the condition \( d\eta\wedge d\eta =0 \). By the above proposition, the associated \KT cylinder has closed \( 3 \)-form torsion. Consider, for instance, the Lie algebra \( \g \) endowed with an orthonormal basis \( \Lel e^1,\ldots,\Lel e^5 \) of \( \g^* \) such that
\begin{gather*}
d\Lel e^1=(\Lel e^1+\Lel e^2+\Lel e^4-\Lel e^5)\wedge\Lel e^3+\Lel e^2\wedge \Lel e^5,\\
d\Lel e^2=(-2\Lel e^2+2\Lel e^3-\Lel e^4+\Lel e^5)\wedge\Lel e^1+\Lel e^4\wedge( \Lel e^2-\Lel e^3+\Lel e^5),\\
d\Lel e^3=(\Lel e^2-\Lel e^3-\Lel e^4+\Lel e^5)\wedge\Lel e^1+\Lel e^4\wedge(-2e^2+2e^3+ \Lel e^5),\\
d\Lel e^4=(-\Lel e^1+\Lel e^3-\Lel e^4+\Lel e^5)\wedge\Lel e^2-\Lel e^3\wedge \Lel e^5,\,\quad  d\Lel e^5=(\Lel e^1-\Lel e^4)\wedge(\Lel e^2-\Lel e^3).
\end{gather*} 
The Lie algebra \( \g \) admits an \ST structure defined by \( g=\sum_{j=1}^5\Lel e_j^2 \) together with \( \xi=\Lel E_5, \eta=\Lel e^5 \) and \( \Lel e^1\circ\vp=-\Lel e^2 \), \( \Lel e^3\circ\vp=-\Lel e^4\). As \( d\eta \) is decomposable, we clearly have \( d\eta\wedge d\eta=0 \).
\end{example}

\subsection{Left-invariant ST structures on Lie groups}
\label{sec:Lie-grp-ST}

It is well known \cite[Theorem 5]{Boothby-W:contact} that the only semi-simple connected Lie groups carrying a left-invariant contact structure are the split form \( \SL(2,\bR) \) and its compact dual \( S^3\cong\SU(2) \). The following result thus illustrates the class of \ST structures is much richer than that of Sasaki structures. Indeed, each odd-dimensional compact Lie group admits an \( \ST \) structure whose torsion \( 3 \)-form is closed, i.e.,
an \( \SST \) structure.

\begin{theorem}
\label{thm:cmp-Lie-grp}
Any odd-dimensional connected compact Lie group admits a left-\break{invariant} \SST structure. 
\end{theorem}
\begin{proof}
Let \( G \) be as in the statement of the theorem. Decompose its complexified Lie algebra as
\[ \g_\bC=\lt\oplus\bigoplus_\alpha \g_\alpha, \]
with  \( \alpha \) ranging in the space of roots, on which we fix an ordering. Now let \( \sigma \) be the real structure given by  conjugation. By construction, \( \sigma(\lt )\subset\lt \) and \( \sigma(\g_\alpha)\subset \g_{-\alpha} \).

Fix a positive definite inner product \( g \) on \( \g=Z(\g)\oplus [\g,\g] \) that reflects the splitting and extends the negative of the Killing form. Then \( \lt \) is orthogonal to each \( \g_{\alpha} \). Pick an almost contact metric structure on \( \lt_\bR \), compatible with the restriction of \( g \). By extending \( \bC \)-linearly, we obtain an endomorphism \( \vp \colon\lt\to\lt \) that commutes with \( \sigma \). This is now extended to a \( \bC \)-linear morphism \( \vp \colon\,\g_\bC\to \g_\bC \) via \[ \vp (\lt)\subset \lt, \quad \vp |_{\g_\alpha}=\begin{cases} i\id, \quad \alpha>0 \\ -i\id, \quad \alpha<0\end{cases}; \]
in particular, note that this extension also commutes with \( \sigma \).

For \( \Lel A\in \g_\alpha \), \( \Lel B \in\g_\beta \), we compute \( S(\Lel A,\Lel B) \) case-by-case as follows:
if \( \alpha,\beta>0 \) then
\begin{equation*}
[\vp(\Lel A),\vp (\Lel B)]-[\Lel A,\Lel B]-\vp [\vp(\Lel A),\Lel B]-\vp [\Lel A,\vp(\Lel B)]=-2[\Lel A,\Lel B]-2\vp [iA\Lel ,\Lel B]=0; 
\end{equation*}
if \( \alpha>0>\beta \) then
\begin{equation*}
[\vp(\Lel A),\vp(\Lel B)]-[\Lel A,\Lel B]-\vp [\vp(\Lel A),\Lel B]-\vp [\Lel A,\vp(\Lel B)]=0;
\end{equation*}
if \( 0>\alpha,\beta \) then
\begin{equation*}
[\vp(\Lel A),\vp(\Lel B)]-[\Lel A,\Lel B]-\vp [\vp(\Lel A),\Lel B]-\vp [\Lel A,\vp(\Lel B)]=-2[\Lel A,\Lel B]+2\vp [i\Lel A,\Lel B]=0.
\end{equation*}
For \( \Lel A\in \g_\alpha \), \( \Lel H\in\lt \), we have: if \( \alpha>0 \) then 
\begin{equation*}
\begin{split}
   S(\Lel H,\Lel A)&=[\vp(\Lel H),\vp(\Lel A)]-[\Lel H,\Lel A]-\vp [\vp(\Lel H),\Lel A]-\vp [\Lel H,\vp(\Lel A)]\\
   &=i \alpha(\vp(\Lel H))\Lel A - \alpha(\Lel H)\Lel A- \alpha(\vp (\Lel H))\vp(\Lel A)+\alpha(\Lel H) \Lel A=0;
 \end{split}
 \end{equation*}
 if \( \alpha<0 \) then
  \begin{equation*}
  \begin{split} 
   S(\Lel H,\Lel A)&=[\vp(\Lel H),\vp(\Lel A)]-[H\Lel ,\Lel A]-\vp [\vp(\Lel H),\Lel A]-\vp [\Lel H,\vp(\Lel A)]\\
   &=-i \alpha(\vp(\Lel H))\Lel A - \alpha(\Lel H)\Lel A- \alpha(\vp(\Lel H))\vp(\Lel A)+\alpha(\Lel H) \Lel A=0. 
\end{split}
\end{equation*}
Finally, \( S(\Lel H,\Lel K)=0 \) holds trivially, by construction, for \( \Lel H,\Lel K\in\lt \).

In conclusion, we have constructed a normal almost contact metric structure \( (g,\vp,\eta,\xi) \). In order to see that this structure is actually \ST, observe that the metric \( g \) is \( \ad(\g) \)-invariant (whereas \( \vp \), \( \xi \) and \( \eta \) are only  \( \ad(\lt) \)-invariant). In particular, this implies \(\Ld_\xi g=0\), meaning \( \xi \) is Killing, as required.

Moreover, by uniqueness (cf. Remark \ref{rem:ST-KT-comp}), the \ST connection must be the canonical connection on \( G \) defined via \( \NB_{\Lel X}\Lel Y=0, \quad \Lel X,\Lel Y\in\g \); clearly, \( \NB \) is an almost contact metric connection and the associated torsion \( c \) is proportional to the closed \( 3 \)-form \( (X,Y,Z)\mapsto g([X,Y],Z) \).
\end{proof}

\begin{remark}
\label{rem:SST-grp}
The proof of the above theorem has a number of consequences. Most notably, we remark:
\begin{compactenum}
\item generally one obtains many inequivalent \ST structures on each Lie group \( G \); this is due to the flexibility in the choice of almost contact metric structure on \( \lt_\bR \).   
\item in addition to being closed, the torsion \( 3 \)-form \( c \) is coclosed. Consequently, these group examples satisfy the second equation of \eqref{eq:field-type-II}.
\item as the \ST connection associated with each of the above \ST structures is flat, the first condition appearing in \eqref{eq:field-type-II}, \( \Ric^\NB=0 \), is obviously satisfied. Moreover, the vanishing of the \ST Ricci form implies a further reduction of the restricted holonomy group, \( \Hol(\NB)\subseteq 1\times\SU(k) \).  
\end{compactenum}
\end{remark}

For many interesting Lie group examples, e.g., the non-Abelian nilpotent ones, the \ST connection is not flat. However, it is still fairly easy to compute \( \NB \) in an efficient way. In order to do this, one uses the familiar relationship between \( d \) and \( \LC \) which, in terms of the isomorphism \( \Phi\colon\,\g\otimes\Lambda^2\g\to\Lambda^2\g\otimes\g \) given via the inclusion followed by wedging, can be put in the form  \( d=\Phi(\LC) \) (cf. \cite[Lemma 3.1]{Salamon:tour}). Regarded as an element in \(  \g^*\otimes\Lambda^2\g^* \), one can then express the \ST connection as \( \NB=\Phi^{-1}(d)+c\slash2 \), where 
\begin{equation*}
 2\,\Phi^{-1}((\Lel e^j\wedge \Lel e^k)\otimes \Lel e^i)=-\Lel e^i\otimes (\Lel e^j\wedge\Lel e^k)+\Lel e^k\otimes(\Lel e^i\wedge \Lel e^j)+\Lel e^j\otimes(\Lel e^k\wedge\Lel e^i),
\end{equation*}
with respect to a chosen \( g \) orthonormal basis \( \{ \Lel e^i \} \) of \( \g^* \). Example \ref{ex:STconn} illustrates the use of this formula.

\subsection{The Lee 1-form}

In \KT geometry, the so-called Lee \( 1 \)-form plays an important role (see, e.g., \cite{Fino-al:SKT}). Almost contact metric geometry also operates with the notion of a Lee form (cf. \cite{Chinea-al:submer}); generally this \( 1 \)-form has a component proportional to \( \eta \), but as in \cite[Section 5]{Friedrich-I:contact}, things can be phrased more naturally if we disregard this term. More precisely, by analogy with the \KT case (see, e.g., \cite{Ivanov-P:vanishing-string}), we define the \emph{Lee \( 1 \)-form \( \vartheta \) of an \ST manifold \( (N^n,g,\xi,\eta,\vp) \)} by \[ \vartheta(X):=-\frac12\sum_{i=1}^nc(\vp(X), E_i,\vp( E_i)),\] where \(  E_1,\ldots, E_n \) is a (local) orthonormal frame of  \( (N,g) \). Independence on the choice of frame follows from the rightmost hand side of the following  expressions:

\begin{proposition}
\label{prop:Lee-form}
On an \ST manifold, the Lee \( 1 \)-form is given by \[ \vartheta(X)=\frac12\sum_{i=1}^ndF(X, E_i,\vp(E_i))=-(d^*F\circ\vp)( X). \]
\end{proposition} 
\begin{proof}
Firstly, we observe that \[ \sum_{i=1}^n\eta\wedge d\eta(\vp(X),E_i,\vp(E_i))=0; \] we are working in an orthonormal frame \( E_i \) adapted to the structure, meaning the dual coframe \( e^i \) satisfies \( F=e^1\wedge e^{2}+\cdots+e^{n-2}\wedge e^{n-1}; \) we can also assume that \( X \) is one of the \( E_i \). Consequently, we get the first equality:
\begin{equation*}
\begin{split}
-\frac12\sum_{i=1}^nc(\vp(X), E_i,\vp( E_i))&=-\frac12\sum_{i=1}^ndF(\vp^2(X), \vp(E_i),\vp^2( E_i))\\
&=\frac12\sum_{i=1}^ndF(X, E_i,\vp(E_i)).
\end{split}
\end{equation*}

Next, we apply the formulae (for the second see \cite{Friedrich-I:spinors}):
\begin{gather*}
(\LC_XF)(Y,Z)=g((\LC_X\vp)Y,Z)\\
g((\LC_X\vp)Y,Z)=-\frac12\left(c(X,Y,\vp(Z))+c(X,\vp(Y),Z)\right)
\end{gather*}
together with the definition of the codifferential to get:
\begin{equation*}
\begin{split}
(d^*F)(X)&=-\sum_{i=1}^n(\LC_{E_i}F)(E_i,X)=-\sum_{i=1}^ng((\LC_{E_i}\vp)E_i,X)\\
   &=\frac12\sum_{i=1}^nc(X,E_i,\vp(E_i)).
   \end{split}
\end{equation*}
From this computation the second equality of the proposition readily follows.
\end{proof}

We shall say that an \ST manifold is \emph{balanced} if the associated Lee \( 1 \)-form is zero, \( \vartheta=0 \). Clearly, any \ST manifold which is quasi-Sasaki is balanced. In low dimensions, the converse is also true. Indeed, \emph{any} \( 3 \)-dimensional \ST manifold is both balanced and quasi-Sasaki; \( dF \) vanishes identically in this case. In dimension \( 5 \), a characterisation follows by using:

\begin{lemma}
\label{lemma:lee-5-ST}
On a \( 5 \)-dimensional \ST manifold, the following relation holds 
\begin{equation}
\label{eq:lee-5-ST}
 dF=\vartheta\wedge F.
 \end{equation}
\end{lemma}
\begin{proof}
Firstly we compute
\[\vartheta(\vp\cdot)=d^\Hodge F \pmod \eta=\Hodge(dF\wedge\eta)\pmod \eta \]
which implies
\[ \vartheta\wedge F=dF, \]
as required.
\end{proof}

As an immediate consequence, we have:

\begin{proposition}
\label{prop:balanced-qS-5dim}
A \( 5 \)-dimensional \ST manifold is balanced if and only if it is quasi-Sasaki.
\qed
\end{proposition} 

As the next example shows, things change from dimension \( 7 \).

\begin{example}
\label{ex:STconn}
Inspired by \cite{Dotti-F:ahcplx}, let us consider the connected nilpotent Lie group 
\( H \) such that the dual  of its Lie algebra has a basis \( \{e^1,\dotsc, e^7\} \) satisfying
\begin{gather*}
d\Lel e^1=0=d\Lel e^2=d\Lel e^3=d\Lel e^4,\quad
d\Lel e^5=-\Lel e^1\wedge\Lel e^2+\Lel e^3\wedge\Lel e^4,\\
d\Lel e^6=-\Lel e^1\wedge\Lel e^3-\Lel e^2\wedge\Lel e^4,\quad
d\Lel e^7=-\Lel e^1\wedge\Lel e^4+\Lel e^2\wedge\Lel e^3.
\end{gather*}
On \( \lie h \) we define an \ST structure by declaring this basis to be orthonormal, \( \eta=\Lel e^7 \), and \[ \Lel e^1\circ\vp=- \Lel e^2,\, \Lel e^3\circ\vp= -\Lel e^4,\, \Lel e^6\circ\vp= -\Lel e^7,\, \Lel e^5\circ\vp= 0. \] Consequently, \( F=\Lel e^1\wedge\Lel e^2+\Lel e^3\wedge\Lel e^4+\Lel e^6\wedge \Lel e^7 \), and \[ dF=\Lel e^6\wedge(\Lel e^1\wedge \Lel e^4-\Lel e^2\wedge \Lel e^3)-\Lel e^7\wedge(\Lel e^1\wedge \Lel e^3+\Lel e^2\wedge \Lel e^4). \] The latter expression implies \( dF(\Lel X,\Lel E_j,\vp(\Lel E_j))=0 \) for all \( 1\leqslant j\leqslant 7 \), so that the \ST structure is balanced.

The torsion \( 3 \)-form associated with the above \ST structure is determined via 
\begin{equation*}
\begin{split}
 c=-\Lel e^1\wedge\Lel e^2\wedge\Lel e^5&-\Lel e^1\wedge\Lel e^3\wedge\Lel e^6-\Lel e^1\wedge\Lel e^4\wedge\Lel e^7\\
 &-\Lel e^2\wedge\Lel e^4\wedge\Lel e^6+\Lel e^2\wedge\Lel e^3\wedge\Lel e^7+\Lel e^3\wedge\Lel e^4\wedge\Lel e^5,
 \end{split}
 \end{equation*}
and we can then find an explicit expression for the \ST connection by using the discussion in the last part of section \ref{sec:Lie-grp-ST}:
\begin{equation*}
\begin{array}{lll}
\NB_{\Lel E_5}\Lel E_1=-\Lel E_2, & \NB_{\Lel E_6}\Lel E_1=-\Lel E_3, & \NB_{\Lel E_7}\Lel E_1=-\Lel E_4,\\
\NB_{\Lel E_5}\Lel E_2=\Lel E_1, & \NB_{\Lel E_6}\Lel E_2=-\Lel E_4, & \NB_{\Lel E_7}\Lel E_2=\Lel E_3,\\
\NB_{\Lel E_5}\Lel E_3=\Lel E_4, & \NB_{\Lel E_6}\Lel E_3=\Lel E_1, & \NB_{\Lel E_7}\Lel E_3=-\Lel E_2,\\
\NB_{\Lel E_5}\Lel E_4=-\Lel E_3, & \NB_{\Lel E_6}\Lel E_4=\Lel E_2, & \NB_{\Lel E_7}\Lel E_4=\Lel E_1.
\end{array}
\end{equation*}
We remark that this \ST connection coincides with the Bismut connection of the \textsc{hkt} structure on \( \bR\times H \) found in \cite[p. 560]{Dotti-F:ahcplx}. 
\end{example}

\subsection{Mappings of ST manifolds}
\label{sec:conf-trans}

It is reassuring to observe that \ST manifolds behave well with respect to certain types of mappings.

\paragraph{Transversal conformal transformations}
The notion of \emph{transversal conformal transformations} (see \cite[Section 5]{Friedrich-I:contact}) fits well into the framework of \ST geometry. Given a basic function \( f \) and an \ST structure \( (g,\xi,\eta,\vp) \) on \( N^n \), we define a new almost contact metric structure via the expressions:
\begin{equation*}
\wt \vp:=\vp,\,\wt\xi:=\xi,\,\wt\eta:=\eta,\,\wt g:=e^{2f}g+(1-e^{2f})\eta^2. 
\end{equation*}
The essential observation is now the following specialisation of \cite[Proposition 5.1, Proposition 5.2]{Friedrich-I:contact}:
\begin{proposition}
The class of \ST manifolds is invariant under transversal conformal transformations. Moreover, the \( 3 \)-form torsion of the transformed \ST structure takes the form
\begin{gather*}
\wt c=c+(e^{2f}-1)d^\vp F+2e^{2f}d^\vp f\wedge F,
\end{gather*}
and the transformed Lee \( 1 \)-form is given by \[ \wt \vartheta=\vartheta+(n-3)df.\] 
\end{proposition}
\begin{proof}
The underlying almost contact structure is unchanged, so normality is preserved. In addition, since \( f \) is basic, we see that \[ \wt \xi\hook d\wt F=\xi\hook(e^{2f}dF)=0. \] So, by Proposition \ref{prop:ST-charac-deriv}, \( (\wt g,\wt\xi,\wt\eta,\wt \vp) \) defines an \ST structure on \( N \).  The expressions for \( \wt c \) follows from a straightforward computation which uses \( F(\vp\cdot,\vp\cdot)=F \). Finally, the expression for the transformed Lee form follows from the computation:
\begin{equation*}
\begin{split}
\wt\vartheta(X)&=\frac12\sum_{i=1}^nd\wt F(X,\wt E_i,\vp(\wt E_i))=\frac12\sum_{i=1}^{n}d(e^{2f} F)(X,\wt E_i,\vp(\wt E_i))\\
    &=\vartheta(X)+\sum_{i=1}^ne^{2f}(df\wedge F)(X,\wt E_i,\vp(\wt E_i))=\vartheta(X)+(n-3)df(X),
\end{split}
\end{equation*}
where the last equality follows by noting that we can assume \( \wt E_i \) is a frame adapted to the structure, as in the proof of Proposition \ref{prop:Lee-form}.
\end{proof}

In analogy with the Hermitian setting, we dub an \ST manifold \emph{conformally balanced} if the associated Lee \( 1 \)-form is exact. 

\begin{proposition}
An \ST structure \( (\wt g,\wt\xi,\wt\eta,\wt\vp) \) on \( N^n \), \( n\geqslant5\), is transversal conformal to a quasi-Sasaki structure if and only if it is conformally balanced and satisfies \[ d\wt F=\frac2{n-3}\wt\vartheta\wedge\wt F, \] and \( (\wt g,\wt\xi,\wt\eta,\wt\vp) \). 
\end{proposition}
\begin{proof}
Assume first that \( (\wt g,\wt\xi,\wt\eta,\wt\vp) \) is transversal conformal to a quasi-Sasaki structure \( (g,\xi,\eta,\vp) \) by the function \( f \). Clearly, \( \vartheta=0 \) so that \( \wt\vartheta=(n-3)df\) which implies that \( (\wt g,\wt\xi,\wt\eta,\wt\vp) \)  is conformally balanced and also \[ d\wt F=d(e^{2f}F)=2df\wedge \wt F=\frac2{n-3}\wt\vartheta\wedge \wt F. \]

Conversely, assume \( (\wt g,\wt\xi,\wt\eta,\wt\vp) \) satisfies the above relation together with the condition \( \wt\vartheta=dh \), for some basic function \( h \). Then the \ST structure obtained via a transversal conformal transformation by \( f=-h\slash{(n-3)} \) is quasi-Sasaki. 
\end{proof}

By specialising to dimension \( 5 \) and using Lemma~\ref{lemma:lee-5-ST}, we get (see also \cite[Proposition 5.4]{Friedrich-I:contact}):

\begin{corollary}
Any \( 5 \)-dimensional conformally balanced \ST manifold is transversal conformal to a quasi-Sasaki manifold.
\qed
\end{corollary}

The above arguments reveal a way of producing \ST manifolds that are not quasi-Sasaki:

\begin{proposition}
\label{prop:Sasak-to-non-Sasak}
In dimensions \( \geqslant5 \), any transversal conformal transformation of a connected quasi-Sasaki manifold by a non-constant function, gives rise to an \ST manifold that is not quasi-Sasaki.  
\end{proposition}
\begin{proof}
Let \( (N^n,g,\xi,\eta,\vp) \) be a quasi-Sasaki manifold, \( n\geqslant5 \). Consider a transversal conformal transformation \( (\wt g,\wt\xi,\wt\eta,\wt\vp) \) by \( f \). In particular, we have that \( \wt F=e^{2f}F \) which implies \[ d\wt F=d(e^{2f}F)=2e^{2f}df\wedge F+e^{2f}dF= 2df\wedge \wt F.\] As the fundamental \( 2 \)-form \( F \) defines (pointwise) an injective map \( T^*_pN\to\Lambda^3T^*_pN \) via wedging: \( \Lel a\mapsto \Lel a\wedge F_p \), we see that \( d\wt F=0 \) if and only if \( df=0 \). Consequently, the transformed \ST structure is quasi-Sasaki if and only if \( f \) is constant, as asserted.
\end{proof}

\begin{example}
\label{ex:Sasak-sphere}
The standard Sasaki structure \( (g,\xi,\eta,\vp) \) on the unit sphere, \[ S^{2k+1}=\left\{z=x+iy\in\bC^{k+1}\colon\,\sum_j|z_j|^2=\sum_j(x_j^2+y_j^2)=1\right\},\] is given as follows. One takes as metric the restriction of the Euclidean metric, \( g=\left.\sum_jdx_j^2+dy_j^2\right|_{S^{2k+1}} \), and the endomorphism \( \vp \) is defined via restriction of the standard complex structure on \( \bR^{2(k+1)}\cong\bC^{k+1} \). The Reeb vector field and its dual \( 1 \)-form are taken to be
\begin{gather*}
\xi=\sum_{j=1}^{k+1}x_j\partial\slash{\partial y_j}-y_j\partial\slash{\partial x_j}\quad \textrm{and} \quad \eta=\sum_{j=1}^{k+1}x_jdy_j-y_jdx_j.
\end{gather*}

Any choice of a basic function \( f \) then gives rise to an \ST structure via a transversal conformal transformation, and, by Proposition \ref{prop:Sasak-to-non-Sasak}, this structure is quasi-Sasaki if and only if  \( f \) is constant. As a concrete example, we can pick \( f \) to be of the form 
\begin{equation*}
f(x,y):=\sum_{j=1}^{k+1}\lambda_j\left(x_j^2+y_j^2\right),
\end{equation*} 
where \( \lambda_j \) are constants.
\end{example}

\paragraph{Transversal homotheties}
There is a well-known notion of \emph{transversal homotheties} in contact geometry \cite{Tanno:homoth}. It carries over to the class of \ST manifolds as follows. Let \( a\in\bR_+ \) be a real positive constant. A \emph{transversal homothety} of \( (N,g,\xi,\eta,\vp) \) by \( a \) is defined by putting \[ \wt g:=ag+a(a-1)\eta^2,\, \wt \xi:=\xi\slash a,\,\wt\eta:=a\eta,\,\wt\vp:=\vp.\]

\begin{proposition}
A transversal homothety transformation of an \ST manifold is an \ST manifold. Moreover, the fundamental \( 2 \)-form, Lee \( 1 \)-form and torsion \( 3 \)-form of the transformed structure are given by 
\[ \wt F=aF,\,\wt\vartheta=\vartheta,\,\textrm{ and }\,\wt c\slash a=(a-1)\eta\wedge d\eta+c, \] respectively.
\end{proposition}
\begin{proof}
A transverse homothety transformation of \( (N,g,\xi,\eta,\vp) \) by \( a\in\bR_+ \) is obviously a normal almost contact structure. The transformed structure clearly has fundamental \( 2 \)-form given by \( \wt F(X,Y)=\wt g(\wt \vp(X),Y)=aF(X,Y) \) so that \( d\wt F=adF \). This implies \( \wt\xi\hook d\wt F=0 \), i.e., the transformed structure is \ST. We also get
\begin{equation*}
\wt\vartheta(X)=\frac12\sum_{i=1}^nd\wt F(X,\wt E_i,\vp(\wt E_i))=\frac12\sum_{i=1}^nd F(X,E_i,\vp( E_i))=\vartheta(X),
\end{equation*}
and
\begin{equation*}
\wt c=\wt\eta \wedge d\wt\eta+d\wt F(\wt\vp\cdot,\wt\vp\cdot,\wt\vp\cdot)=a^2\eta\wedge d\eta+ad^\vp F=a\left((a-1)\eta\wedge d\eta+c\right),
\end{equation*}
as claimed.
\end{proof}

Note that transversal homotheties preserve the properties of being quasi-Sasaki and balanced. Also observe that if one starts from an \SST manifold satisfying the additional requirement \( d\eta\wedge d\eta=0 \) (see Example \ref{ex:quasi-Sasak-SKT}), then one can obtain a \( 1 \)-parameter family of \SST structures via transversal homotheties.

\paragraph{Riemannian submersions}
Following \cite{Chinea:submer}, we shall use the terminology \emph{almost contact metric submersion} to denote a Riemannian submersion \[ \pi\colon\,(N,g,\xi,\eta,\vp)\to(\wt N,\wt g,\wt\xi,\wt\eta,\wt\vp) \] which is also an almost contact mapping, meaning \( \pi_*\circ\vp=\wt\vp\circ\pi_* \); the vector field \( \xi \) is horizontal in this case, and we shall assume \( \wt \xi=\pi_*(\xi) \). The class of \ST manifolds behaves well with respect to this type of mappings:

\begin{proposition}
Let \( \pi\colon\,(N^n,g,\xi,\eta,\vp)\to(\wt N^{\wt n},\wt g,\wt\xi,\wt\eta,\wt\vp) \) be an almost contact metric submersion. If \( N \) is an \ST manifold, then so is \( \wt N \). In that case, the torsion \( 3 \)-forms are related via \[ \wt c(\wt X,\wt Y,\wt Z)\circ \pi=c(X,Y,Z), \] where \( X \) is the the basic vector field on \( N \) corresponding to \( \wt X \), and so forth. Moreover, the fibres of \( \pi \) are invariant Hermitian submanifolds of \( N \) of dimension \( n-\wt n \).
\end{proposition}
\begin{proof}
The first assertions are immediate consequences of the following identities of \cite[Proposition 1.1]{Chinea-al:submer}: 
\begin{gather*}
\wt S(\wt X,\wt Y)\circ\pi=\pi_*(S(X,Y)),\,\wt F(\wt X,\wt Y)\circ\pi=F(X,Y),\\
d\wt F(\wt X,\wt Y,\wt Z)\circ\pi=dF(X,Y,Z),\,\wt \eta(\wt X)\circ\pi=\eta(X),\,\wt d\eta(\wt X,\wt Y)\circ\pi=d\eta(X,Y).
\end{gather*}

The claim regarding the fibres follows directly from \cite[Proposition 2.1, Theorem 2.1]{Chinea:submer}.
\end{proof}

\begin{example}
Let \( (M,g_M,J) \) be a \KT manifold and \( (\wt N,\wt g,\wt \vp) \) an \ST manifold. Then we can endow the product \( N=M\times \wt N \) with an \ST structure in the obvious way; one takes \( g:=g_M+\wt g,\,\eta:=\wt\eta,\,\xi:=\wt\xi,\,\vp:=J+\wt\vp \). An almost contact metric submersion is now given by projecting onto the second factor, i.e., \( \pi\colon\,N\to \wt N \) has \( \pi(x,\wt x)=\wt x \). The fibres are obviously copies of \( M \).
\end{example}

\begin{remark}
Conceivably, the above observations could play a role in the study of harmonic morphisms \cite{Baird-W:harmonic}. In particular, one could follow the ideas of \cite{Chinea:harmonic} which studies horizontally conformal  \( (\vp ,\wt\vp) \)-holomorphic submersions; a submersion\linebreak \( \pi\colon\,(N,g,\xi,\eta,\vp)\to(\wt N,\wt g,\wt\xi,\wt\eta,\wt\vp) \) is an almost contact mapping satisfying in addition \( \wt g(\pi_*(X),\pi_*(Y))=\lambda^2g(X,Y) \), for any horizontal vector fields \( X,Y \), and with \( \lambda \) denoting a smooth nowhere vanishing function. In particular, note that if \( \lambda=1 \) then \( \pi \) is an almost contact metric Riemannian submersion.  
\end{remark}

\section{Interpolating between ST and KT structures}
\label{sec:sandwich}

One salient feature of Sasaki geometry is its relation to K\"ahler manifolds (see, for instance, \cite[Chapter 6]{Boyer-G:Sasak-book}). In the setting of torsion geometry, one may similarly ask whether the concepts of \ST and \KT manifolds are related. 

\paragraph{The KT cylinder over an ST manifold}
\label{sec:ST-KT}

Given an \ST manifold  \( (N,g,\xi,\eta,\vp) \) we consider the cylinder \( \mathcal K(N):=\bR\times N \) which has \( N \) as its base. We give this the product metric \( g_{\mathcal K}:=ds^2+g \) and define an almost complex structure \( J_{\mathcal K} \) on \( \mathcal K(N) \) by requiring that \( \omega_{\mathcal K}=g(J_{\mathcal K}\cdot,\cdot) \), where by definition \( \omega_{\mathcal K}= ds\wedge\eta+F \).  We dub \( (\mathcal K(N),g_{\mathcal K},J_{\mathcal K}) \) the \emph{cylinder associated with \( (N,g,\xi,\eta,\vp) \)}. By specializing the arguments  in the proof of \cite[Theorem 2.3]{Fernandez-al:SKT}, we find:

\begin{proposition}
\label{prop:ST-KT-cyl}
The cylinder \( \mathcal K(N) \) associated with an \ST manifold \( (N,g,\xi,\eta,\vp) \) is a \KT manifold. Moreover, \( \mathcal K \) is 
\SKT (resp. balanced) if and only if \( N \) is \textsc{sst} (resp. balanced).
\end{proposition}
\begin{proof}
Write \( n=2k+1\). At a given point \((s,p)\in \bR \times N \), we may pick an oriented orthonormal coframe \( \Lel f,\Lel e^n,\Lel e^1,\ldots \) such that \[ \omega_{\mathcal K}=\Lel f\wedge\Lel e^{n}+\sum_{i=1}^k\Lel e^{2i-1}\Lel \wedge\Lel e^{2i},\] where \( \Lel e^n=\eta \) and \( \Lel f=ds \). It is well known (cf. \cite{Sasaki-H:Nijenhuis})  that the normality of \( (g,\xi,\eta,\vp) \) is equivalent to the integrability of the compatible almost complex structure \( J_{\mathcal K} \). In particular, this means that \( ({\mathcal K}(N),g_{\mathcal K},J_{\mathcal K}) \) is \KT. 

Next, in order to express the associated totally skew-symmetric torsion term, let us consider the \( 3 \)-form \( d\omega_{\mathcal K}=-ds\wedge d\eta+dF \) . At the given point, we consider the decomposition \( dF=\Lel a+\Lel b\wedge\Lel e^{n+1} \), where \( \Lel a,\Lel b\in \Lambda^*\Span{\Lel e^1,\ldots,\Lel e^{2k}}. \) As \( \xi\hook dF=0 \), by Proposition \ref{prop:ST-charac-deriv}, we must have \( \Lel b=0 \). Consequently, we find \[ dF(J_{\mathcal K}\cdot,J_{\mathcal K}\cdot,J_{\mathcal K}\cdot)=\Lel a(J_{\mathcal K}\cdot,J_{\mathcal K}\cdot,J_{\mathcal K}\cdot)=\Lel a^\vp=d^\vp F. \]

In a similar way, and by using the fact that normality ensures that \( d\eta(\vp\cdot,\vp\cdot)=d\eta(\cdot,\cdot) \), we obtain \( d\eta(J_{\mathcal K}\cdot,J_{\mathcal K}\cdot)=d\eta \).

In summary, we can now express the torsion \( 3 \)-form associated with\break{} \( (N,g,\xi,\eta,\vp) \) via
\[ d\omega_{\mathcal K}(J_{\mathcal K}\cdot,J_{\mathcal K}\cdot,J_{\mathcal K}\cdot)=\eta\wedge d\eta+d^\vp F. \]
The expression for the torsion \( 3 \)-form  of \( ({\mathcal K}(N),g_{\mathcal K},J_{\mathcal K}) \) immediately implies the last assertion.
\end{proof}

\begin{remark}
The \KT cylinder may be viewed as a special case of a more general product construction. Given two \ST manifolds  \( (N_\pm,g_\pm,\xi_\pm,\eta_\pm,\vp_\pm) \), we can form the product \( \mathcal P(N_+\times N_-) \) which we equip with the metric \( g_\mathcal P=g_++g_- \) and almost complex structure \( J_{\mathcal P} \) defined via the compatibility condition \( g(J_{\mathcal P}\cdot,\cdot)=\omega_{\mathcal P } \), where \( \omega_{\mathcal P}=\eta_-\wedge\eta_++F_-+F_+ \). It is well-known \cite[Proposition 3]{Morimoto:acm-normal} that the conditions \( S_+=0=S_- \) ensure the integrability of \( J_{\mathcal P} \). Consequently, \( (\mathcal P(N_+\times N_-),g_{\mathcal P}, J_{\mathcal P}) \) is a \KT manifold. 
\end{remark}

\begin{example}
Starting from a left-invariant \SST structure on a compact Lie group \( G \) (cf. Theorem \ref{thm:cmp-Lie-grp}), one obtains an \SKT structure on the cylinder \( \mathcal K( G) \); these examples are well known \cite{Spindel-al}. In fact, it follows, by Remark \ref{rem:SST-grp}, that \( \mathcal K( G) \) is an \( \textsc{scyt} \) manifold, meaning, in addition to being \SKT, that the Bismut connection has restricted holonomy group contained in \( \SU(k+1) \) (cf. \cite{Grantcharov-G-P:CYT}). 
\end{example}

An alternative to the cylinder construction is the cone construction, where one considers the Riemannian cone \( \mathcal C(N):=\bR_+\times N,g_\mathcal C:=dr^2+r^2g \). \( \mathcal C(N) \) can be equipped with the almost complex structure \( J_{\mathcal C} \) fixed by imposing the compatibility condition \( g_{\mathcal C}(J_{\mathcal C}\cdot,\cdot)=\omega_{\mathcal C} \), where \( \omega_{\mathcal C}=r dr \wedge\eta+r^2F \). In fact, it is possible to characterize \ST structures in terms of \KT structures on the corresponding cone (see \cite{Houri-al:ST} and Remark~\ref{rem:STvsST}).

However, the cone construction behaves less naturally in other respects. Indeed, arguments similar to those in the proof of Proposition \ref{prop:ST-KT-cyl} (see also \cite[Theorem 3.1]{Fernandez-al:SKT}) yield the following:

\begin{proposition}
\label{prop:ST-KT-cone}
The cone associated with an \ST manifold is a \KT manifold. Moreover, the cone is \SKT if and only if the \ST manifold is Sasaki.  
\end{proposition}
\begin{proof}
By construction, we have \( dr(J_{\mathcal{ C}}\cdot)=-r\eta \), and then the usual  \( \vp \)-invariance properties of \( F \) and \( d\eta \) imply
\begin{multline*}
c_{\mathcal C}=d\omega_{\mathcal C}(J_{\mathcal C}\cdot,J_{\mathcal C}\cdot,J_{\mathcal C}\cdot)\\
= rdr(J\cdot)\wedge (-d\eta +2F)+ r^2d^\vp F=r^2(\eta\wedge(d\eta -2F)+ d^\vp F).
\end{multline*}
Consequently, the \SKT condition for the cone is equivalent to the relation
\begin{equation}
\label{eq:ST-cone}
\eta\wedge(d\eta -2F)+ d^\vp F=0. 
 \end{equation}

Clearly, \eqref{eq:ST-cone} is satisfied if the \ST manifold is Sasaki; in that case one has \( d(2F)=d(d\eta)=0 \) and \( \eta\wedge d\eta=\eta\wedge(2F) \). Conversely, suppose the cone is \SKT. Then we find \[ 0= \xi\hook(\eta\wedge(d\eta -2F)+ r^2d^\vp F)=d\eta-2F,\] so that the \ST manifold is Sasaki. 
\end{proof}

\begin{remark}
\label{rem:STvsST}
In \cite{Houri-al:ST}, the cone construction is used to give a different, equivalent definition of an \ST manifold; it is defined as a Riemannian manifold \( (M,g) \) endowed with a \( 3 \)-form \( T \) such that the cone \( \mathcal C(M) \) is \KT with torsion equal to \( r^2T \). By \cite{Houri-al:ST}, this condition implies that the induced almost contact metric structure is normal, and the Reeb vector field is Killing. Moreover, the \( 3 \)-form \( T \) necessarily satisfies
\begin{equation}
 \label{eqn:HourisTorsion}
T=d^\vp F+d\eta\wedge \eta-2F\wedge\eta.
\end{equation}

Conversely, if \( (M,g,\eta,\xi,\vp) \) is \ST in the sense of Definition~\ref{def:ST}, define \( T \) by \eqref{eqn:HourisTorsion}; 
then \( d^T\eta \), defined as \( d\eta-\xi\hook T \), coincides with \( 2F \), and the metric connection with torsion \( T \) satisfies 
\[\nabla^T_X F= \nabla_X F+\frac12 \sum_a (T-c)(X,\Lel E_a,\cdot)\wedge F(\Lel E_a,\cdot)=\eta\wedge X^\flat,\]
so by \cite[Proposition II.2]{Houri-al:ST} the structure is \ST as defined therein. 

Notice that, by \eqref{eqn:HourisTorsion} and consistency with Proposition~\ref{prop:STcharac}, the connection \( \nabla^T \) is not an almost contact metric connection.
\end{remark}

The above cylinder and cone constructions are special instances of warped products. In general, the warped product of an \ST manifold with $\bR$ is \KT, but it can only be \SKT if it is a cylinder or a cone:
\begin{proposition}
\label{prop:warpedprod}
Let \( (N,g,\xi,\eta,\vp) \) be an \ST manifold, and let \( f \) be a non-vanishing function on a connected interval \( I\subset\bR \). Then the warped product metric 
\begin{equation}
\label{eq:warpedmetr}
dr^2+f(r)^2g, 
\end{equation}
on \( I\times N \), is a \KT metric. Moreover, the torsion associated with this \( \KT \) structure is given by
\begin{equation}
\label{eq:warpedtorsion}
f^2(c-2f' F\wedge \eta). 
\end{equation}

In particular, the \KT structure is \SKT only if  \( f \) is constant 
and \( c \) is closed, or \(  f(r)=2\lambda r \) and \(  d\eta=2\lambda F \) for a constant \( \lambda \).
\end{proposition}
\begin{proof}
To see that the warped product \eqref{eq:warpedmetr} is a \KT  metric, we observe
that if \( e^1,\ldots, e^n \) is an  adapted coframe for the given \ST structure, then
\( fe^1, ..., fe^n, dt \) is an adapted coframe of an almost Hermitian structure on 
\( I\times N \). Next, observe that the underlying almost complex structure is compatible with the 
coframe  \( e^1, ..., e^n, \frac1f dt \).
As \( \frac1fdt \)  is a \( 1 \)-form on \( I \), we may write it as \( dh \) for a suitable function \( h\in C^\infty(I) \). With respect to this new coordinate, the coframe \( e^1, ..., e^n, dh \) corresponds to the product almost complex structure, which is known to be integrable.

Regarding the \SST condition, we note that the argument in the proof of Proposition \ref{prop:ST-KT-cone} 
tells us that the \KT torsion is indeed given by \eqref{eq:warpedtorsion}. In particular, if this is closed then, for each \( r \in I \), either \( f'(r)=0 \) and \( c \) is closed, or 
\( d\eta=2f'(r)F \). If the latter condition holds for some \( r \) then \( f'(r) \) must be constant.
\end{proof}
 
\paragraph{A local classification} 
\label{sec:sandwich-loc}

We dub an \ST structure \emph{regular} if the Reeb foliation is regular. In this case the space of leaves is a manifold. The following result describes the structure on this space of leaves. Since regularity always holds locally, it gives a local classification of \ST structures.

\begin{proposition}
\label{prop:sandwich}
Given a regular \ST manifold  \( (N,g,\xi,\eta,\vp) \) then the space of leaves, \( M \), has a unique  Hermitian structure \( (h,J) \)  and  closed \( (1,1) \)-form \( \sigma \) such that 
\begin{compactenum}
\item the projection \( \pi\colon (N,g)\to (M,h) \) is a Riemannian submersion;
\item \(\pi_*\circ \vp  = J\circ \pi_*\); 
\item \( \frac{1}{2\pi}d\eta=\pi^*\sigma \).
 \end{compactenum}

Conversely, given a Hermitian manifold \( (M,h,J) \) with a closed integral \( (1,1) \)-form \( \sigma \), a circle bundle \( N\to M \) with first Chern class equal to \( [\sigma] \), and a connection form \( \eta \) on \( N \) satisfying (iii), then \( N \) has a unique  \ST structure  \( (g,\xi,\eta,\vp) \)  such that \( \xi \) is the fundamental vector field, and the conditions (i)--(iii) are satisfied.

Moreover, every \ST manifold is locally of this form.
\end{proposition}
\begin{proof}
Since \( \xi \) is Killing, there is a unique metric \( h \) satisfying (i). By using \(\Ld_\xi \vp=0,\)
and the fact that \[ \pi_*\colon \ker\eta\to T_{\pi(p)}M\] is (pointwise) an isomorphism,
we deduce that there is a unique almost-complex structure \( J \) satisfying (ii).  If \( X,Y \) are vector fields on \( M \), \( \pi \)-related respectively to \( \wt X,\wt Y \), then (ii) implies that \( J(X) \), \( J(Y) \) are \( \pi \)-related to \( \vp(\wt X)\), \( \vp(\wt Y) \). Consequently, \( N_J(X,Y) \) is \( \pi \)-related to \( S(\wt X,\wt Y)=0 \). This proves \( (h,J) \) is a  Hermitian structure.

The differential form \( d\eta \) is basic, hence it is the pullback of some closed form \( \sigma \). In addition, the fact that \( d\eta \) is of type \( (1,1) \) with respect to \( \vp \) implies \( \sigma \) is \( J \)-invariant, and hence of type \( (1,1) \).

\smallskip
Now let \( (M,h,J) \) be a Hermitian manifold with a closed integral \( (1,1) \)-form \( \sigma \). Condition (i) determines the metric \( g \) on the distribution \( \ker\eta \); since \( \xi \) is orthogonal to this distribution, with unit norm, the metric \( g \) is determined. By construction, the vector field \( \xi \) is Killing. Similarly, condition (ii) determines \( \vp \). In order to prove that  \( S=0 \) identically, it suffices to show that 
\[\pi_*(S(\Lel X,\Lel Y))=0= \eta(S(\Lel X,\Lel Y)), \quad \textrm{for all }\Lel X,\Lel Y\in T_pN .\]
The first equation follows by the same argument as in the first part of the proof whilst the second follows from \( d^\vp\eta=d\eta \), which is a consequence of the hypothesis that \( \sigma \) has type \( (1,1) \).

The last part of the statement follows from the fact that every point of an \ST manifold has a foliated neighbourhood \( \wt N=\wt M\times(0,\epsilon) \). The cohomology class \( [\sigma] \) is then zero, and \( \wt N \) can be identified with an open subset of \mbox{\( \wt M\times S^1 \)}.
\end{proof}

\begin{remark}
The \ST structures of Proposition~\ref{prop:sandwich}  are not contact metric structures in general: if we denote by \( \omega \) the K\"ahler form on the base, \( \omega(X,Y)=h(JX,Y) \), the almost-contact metric structure on \( N \) is a contact metric structure if and only if  \( \sigma=-\omega\slash2\pi \). In particular, this construction can produce \ST manifolds that are not Sasaki, even when the base is K\"ahler.

Applying the construction to non-K\"ahler Hermitian manifolds gives rise to examples of \ST manifolds that are not quasi-Sasaki.
\end{remark}

\begin{remark}
It follows from the local classification of Proposition~\ref{prop:sandwich} that the cone over an \ST manifold can be locally identified with \( \mathcal C(\wt N)=\wt M\times\bC^* \). Slightly more generally, one can consider the case of a circle bundle; then the cone sits inside the complex line bundle with first Chern class equal to \( [\sigma] \).
\end{remark}

\begin{proposition}
In the correspondence of Proposition~\ref{prop:sandwich}, the torsion \( 3 \)-form \( c \) of  \( (N,g,\xi,\eta,\vp) \)  and the one, \( \wt c \), of the Hermitian manifold \( (M,h,J) \) are related by 
\[c=2\pi \eta\wedge(\pi^*\sigma) + \pi^*\wt c. \]
In particular, a \KT manifold is the space of leaves of an \SST manifold if and only if 
\[d\wt c=-4\pi^2 \sigma^2,\]
where \( \sigma \) is a  closed, integral form of type \( (1,1) \).
\end{proposition}
\begin{proof}
By construction \( F=\pi^*\omega \), where \( \omega \) is the K\"ahler form, and we therefore have
\[ d^\vp F=(\pi^* d\omega)(\vp\cdot, \vp\cdot, \vp\cdot)=\pi^*(d\omega(J\cdot,J\cdot,J\cdot))=\pi^*\wt c. \]
\end{proof}

Given an \SST manifold, there are two distinct cases: either \( d\eta \) is decomposable (see Example \ref{ex:quasi-Sasak-SKT} for a \( 5 \)-dimensional example), or \( d\eta\wedge d\eta\neq0 \) (e.g., on Sasaki manifolds). In consistency with the cylinder construction (Proposition \ref{prop:ST-KT-cyl}), the above proposition shows that the study of the decomposable case may be reduced to the study of \SKT manifolds, at least locally. On the other hand, the remaning case cannot be reduced to \SKT geometry.

\section{ST reductions}
\label{sec:ST-red}

Hamiltonian reduction plays an important role in symplectic as well as K\"ahler geometry. There is also a well known reduction of Sasaki manifolds \cite{Grantcharov-O:Sasak-red}, which has been used as a tool for constructing new examples. In this section, we describe one possible way of defining a reduction procedure of \ST manifolds. To some extent, it may be thought of as a generalization of the Sasaki reduction in the same way as Joyce's hypercomplex and quaternionic quotients \cite{Joyce:hclpx-quat-quo} generalise the hyper- and quaternionic-K\"ahler quotient constructions. 
 
\paragraph{A quotient construction}
In the following, we consider an \ST manifold  \( (M,g,\xi,\eta,\vp) \) equipped with a free action of a compact Lie group \( G \) that preserves \ST structure. We denote by \( \Lel X^* \) the fundamental vector field associated to an element \( \Lel X\in\lie{g} \). By imitating \cite{Grantcharov:Reduction}, we have:

\begin{definition}
\label{def:G-moment}
 A \emph{ \( G \)-moment map} is an equivariant mapping \( \mu\colon M\to \g^* \) satisfying the following conditions:
\begin{compactenum}
\item For each non-zero \( \Lel X\in\g \), \( d\mu_p(\vp(\Lel X^*)) \) is nowhere-zero for \( p\in\mu^{-1}(0) \);
\item \( \mu \) is basic with respect to the Reeb foliation.
\end{compactenum}
\end{definition}

The condition (i) ensures that zero is a regular value of \( \mu \), since the linear map
\[\lie{g}\to\lie{g}^*, \quad \Lel X\mapsto d\mu_p(\vp(\Lel X^*))\]
has trivial kernel. It follows \( M_0=\mu^{-1}(0) \) is a regular submanifold, and we shall denote by \( \iota\colon M_0\hookrightarrow M \) the inclusion. As \( G \) acts freely on \( M_0 \), the quotient \( M\sslash G=M_0\slash G \) is also smooth, and we therefore have a principal bundle \( \pi\colon M_0\to M_0\slash G \).

More generally, one can consider a reduction \( \mu^{-1}(\Lel a)\slash G \), where \( \Lel a \) is an arbitrary point in \( (\g^*)^G \). In order for this to work, one needs a stronger definition of \( G \)-moment map. We say a \( G \)-moment map \( \mu \) is \emph{global} if \( d\mu_p(\vp(\Lel X^*)) \) is globally non-zero for each non-zero  \( \Lel X\in\g \). Notice, however, a global \( G \)-moment map has no critical points and therefore can only exist if \( M \) is non-compact.

By construction, we have \( T_pM_0=\ker d\mu_p \), and the distribution
\begin{multline*}
\Hor=\{(p;\Lel X)\in TM_0\colon\, d\mu_p(\vp(\Lel X))=0\}\\
=\{(p;\Lel X)\in TM\colon\, \mu(p)=0=d\mu_p(\Lel X)=d\mu_p(\vp(\Lel X))\}
 \end{multline*}
defines a connection on the principal bundle. Indeed, \( \Hor \) is \( G \)-equivariant because so are \( \mu \) and \( \vp \). In addition, the
condition (i) ensures \( \Hor \) is transversal to \( \ker\pi \). Consequently, a vector field \( X \) on \( M_0 \) has a horizontal projection \( X^{\Hor} \).

We introduce a similar notation for forms, meaning \( \eta^\Hor \) will be the \( 1 \)-form mapping \( X \) to \( \eta(X^\Hor) \), and so forth. The failure of the distribution \( \Hor \) to be orthogonal to \( \ker\pi_* \) is measured by an invariant one-form:
\[\alpha\in\Omega^1(M_0,\lie{g}^*), \quad \langle \alpha(X), \Lel Y\rangle = g(X^\Hor,\Lel Y^*), \quad Y\in\lie{g}.\]
The contraction of \( \alpha \) with the curvature of the principal bundle, regarded as a \( 2 \)-form \( R\in\Omega^2(M_0,\g) \), will play a role when comparing the torsion of \( M \) with that of \( M\sslash G \). Another relevant contraction is the four-form
\[(R\hook c)(X_1,X_2,X_3,X_4)=\sum_{i<j}(-1)^{i+j} c\left(R(X_i,X_j)^*,X_1,\dotsc, \widehat{X_i}, \dotsc, \widehat{X_j},\dotsc, X_4\right).\]

A straightforward, but important, consequence of condition (ii) is the following:

\begin{lemma}
\label{lemma:reduction}
The distribution \( \Hor \) is invariant under \( \vp \) and contains \( \xi \).
\qed
\end{lemma}

Since \( G \) acts preserving the almost contact metric structure, we can define a structure on the quotient using the splitting \( T_pM_0=\Hor_p\oplus\g \) together with projection. Explicitly, we define a Riemannian metric on the quotient by \[ \wt g(\pi_{*p}(\Lel X), \pi_{*p}( \Lel Y))=g(\Lel X,\Lel Y), \quad \Lel X,\Lel Y\in \Hor_p.\] Be warned that for non-zero \( \alpha \), \( \pi \) is generally not a Riemannian submersion (see, for instance, Example~\ref{ex:threespheres}).

Similarly, we define a unit vector field \( \wt\xi \) on \( M_0\slash G \) that is \( \pi \)-related to \( \xi \), and set
\[\wt\eta=\wt\xi^\flat, \quad \wt\vp (\pi_{*p}(\Lel X))=\pi_{*p}(\vp(p;\Lel X)), \quad \Lel X\in \Hor_p.\]
Well-definedness follows from the  \( G \) equivariance of \( \vp \)  and Lemma~\ref{lemma:reduction}.

\begin{proposition}
\label{prop:reduction}
Let \( (M,g,\xi,\eta,\vp) \) be an almost contact metric manifold endowed with the free action of a compact Lie group \( G \) preserving the structure. If there exists a \( G \)-moment map \( \mu\colon M\to\g^* \)
then the reduction \( M\sslash G \) inherits an almost contact metric structure \( (\wt g, \wt\vp, \wt\xi,\wt\eta) \), and if \( M \) is normal so is \( M\sslash G \).

Moreover, if \( M \) is \ST then \( M\sslash G \) is also \ST. In this case, the torsion forms are related via
 \[\pi^*\wt c = c^\Hor - \langle \alpha,R\rangle, \quad \pi^*d\wt c = (dc-R\hook c)^\Hor - \langle (d\alpha)^\Hor,R\rangle.\]
\end{proposition}
\begin{proof}
We have to prove \( (\wt g,\wt \xi,\wt\eta,\wt \vp) \) satisfies \eqref{eq:acm-con}. By Lemma~\ref{lemma:reduction}, these equations reduce to the analogous equations for \( \xi,\eta \) and \( \vp \) on \( \Hor_p \).

Now suppose \( M \) is normal. If \( X,Y \) are vector fields on \( M_0 \) contained in \( \Hor \), then the vanishing of \eqref{eq:Sasaki-Htensor} implies
\[\vp([\vp( X), Y]+ [ X,\vp( Y)])_p\in T_pM_0.\]
Since  \( [\vp( X), Y]_p \) and \([ X,\vp( Y)]_p\) are also in \( T_pM_0 \), it follows that 
\begin{equation}
\label{eqn:normalred}
[\vp( X), Y]_p+ [ X,\vp( Y)]_p\in \Hor_p.
\end{equation}
Therefore, if \( X \) is \( \pi \)-related to \( \wt X \) and \( Y \) is \( \pi \)-related to \( \wt Y \),
\[\pi_{*p}\bigl(\vp\bigl([\vp( X), Y]+ [ X,\vp( Y)]\bigr)\bigr) =\wt\vp([\wt\vp(\wt X),\wt Y]) +\wt\vp([\wt X,\wt\vp(\wt Y)])  .\]
It follows that \( \wt S(\wt X,\wt Y) \) is \( \pi \)-related to \( S(X,Y) \), hence zero.

Suppose in addition that \( M \) is \ST. By the normality assumption we have \( \Ld_\xi\phi=0 \) which implies \( \xi \) preserves \( \Hor \). Indeed, if \( X \) is contained in \( \Hor \), then
\[d\mu([\xi,X])=-d^2\mu(\xi,X)+\Ld_\xi d\mu(X)-\Ld_X d\mu(\xi)=0. \]
By the same token, we have \( d\mu(\vp [\xi,X])=d\mu([\xi,\vp X])=0 \).

Now, using the fact that \( \xi \) is a Killing vector field, we find
\begin{multline*}
\pi^*\left( \Ld_{\wt\xi} \wt g(\wt X,\wt Y)-\wt g([\wt \xi, \wt X], \wt Y)-\wt g(\wt X,[\wt \xi, \wt Y])\right)\\
=\Ld_\xi g(X,Y)-g([\xi,X]^\Hor,Y)-g(X,[\xi,Y]^\Hor)=0.
\end{multline*}

Finally, in order to compute the torsion, we observe \eqref{eqn:normalred} together with the fact that \( \xi \) preserves \( \Hor \) imply that whenever \( X,Y \) are contained in \( \Hor \) then so is  \( [X,Y]-[\vp (X),\vp (Y)] \). Writing \( X=X^\Hor+X^\Ver \), we can rephrase this fact in terms of the equality
\[[X,Y]^\Ver=[\vp(X),\vp(Y)]^\Ver.\]
We also have
\[\pi^*\wt\eta(\wt X)=\eta(X), \quad \pi^*\wt F(\wt X,\wt Y)=F(X,Y),\]
and then compute
\begin{gather*}
\pi^*d\wt\eta(\wt X,\wt Y)=d\eta(X,Y)+\eta([X,Y]^\Ver),\\
\pi^*d\wt F(\wt X,\wt Y,\wt Z)=dF(X,Y,Z)+\sumcic_{\{X,Y,Z\}} F([X,Y]^\Ver,Z),
\end{gather*}
where the summation \( \sumcic \) is a cyclic summation over \( X,Y,Z \). Now we see that
\begin{multline*}
\pi^*d\wt F(\wt\vp(\tilde X),\wt\vp(\wt Y),\wt\vp(\wt Z))=d^\vp F( X, Y, Z)+\sumcic_{\{X,Y,Z\}}g\left([\vp (X),\vp (Y)]^\Ver,-\vp^2 (Z)\right)\\
=d^\vp F(X, Y, Z)+\sumcic_{\{X,Y,Z\}}g\left([X,Y]^\Ver,Z\right)-\sumcic_{\{X,Y,Z\}}\eta\left([X,Y]^\Ver\right)\eta(Z).
\end{multline*}
Summing up gives
\[\pi^*\wt c(\wt X,\wt Y,\wt Z)=c(X,Y,Z)+\sumcic_{\{X,Y,Z\}}g([X,Y]^\Ver,Z)=c(X,Y,Z)-\langle \alpha,R\rangle(X,Y,Z),
\]
as required. The final formula is obtained applying the Bianchi identity.
\end{proof}

\begin{remark}
\label{rem:STred}
In the proof of Proposition~\ref{prop:reduction}, it is not sufficient to assume that \( \xi \) is Killing to prove that \( \wt\xi \) is Killing. Normality is also required. This gives additional motivation for our definition of \ST structures, showing that it is the correct type of structure to consider if one wants the same type of structure to be induced on the reduction.
\end{remark}

\begin{remark}
\label{remark:STKTreduction}
By definition, a \( G \)-moment map is always basic with respect to the Reeb foliation. Thus, if \( M \) is a regular \ST manifold, a \( G \)-moment map on  \( M \) descends to a \( G \)-moment map on the space of leaves \( N \). Moreover, the space of leaves of \( M\sslash G \) can be identified with \( N\sslash G \) as a \KT manifold.
\end{remark}

\begin{remark}
\label{rem:moment-map-cand}
The natural candidate for a \( G \)-moment map is
\[\mu=\sum \Lel E_a\otimes \iota_a\eta.\]
 By definition, this is a \( G \)-moment map only when the matrix
\[\left(d\eta_p(\Lel E_a^*,\vp(\Lel E_b^*)\right))_{ab}\]
is  non-degenerate at each point \( p\in\mu^{-1}(0) \).

Note that  in the Sasaki case, as \( d\eta=2F \), the non-degeneracy condition is equivalent to asserting the fundamental vector fields induced by the action  do not vanish on \( \mu^{-1}(0) \). 
\end{remark}

\paragraph{Non-existence}
The existence of a \( G \)-moment map is a non-trivial topological condition. Consider, for example, the compact Lie group \( N=\Un(3) \) endowed with an \SST structure \( (g,\eta,\xi,\vp) \) as in Theorem~\ref{thm:cmp-Lie-grp}. Denoting by \( \Lel E_{ij} \) the standard basis of the space of complex \( 3\times3 \) matrices, we can assume, at the Lie algebra level, \[ \xi_e=i\Lel E_{11}, \quad \vp_e(i\Lel E_{22})=i\Lel E_{33}.\] Now let \( G \) be the subgroup generated by \( i\Lel E_{22} \), acting on \( \Un(3) \) on the right. Since \( G \) is contained in the maximal torus, it preserves the structure.

Suppose \( \mu\colon \Un(3)\to\g^*\cong\bR \) is a \( G \)-moment map, and consider the tori \( \Un(1) \), \( T_2 \), \( T_3 \) with Lie algebras
\[\lt_{1}=\Span{i\Lel E_{33}}, \quad \lt_2=\Span{i\Lel E_{11},i\Lel E_{22}}, \quad \lt_3=\Span{i\Lel E_{11},i\Lel E_{22},i\Lel E_{33}}, \] respectively. Then \( \mu^{-1}(0)\subset \Un(3) \) is a \( T_2 \) invariant regular submanifold. Denoting its quotient by \( P \), we obtain a diagram
\[\begin{diagram}
 \node{\mu^{-1}(0)} \arrow{e}\arrow{s,l}{T_2} \node {\Un(3)} \arrow{s,l}{T_2}\\
\node{P} \arrow{e} \node {\Un(3)/T_2}\arrow{e,t}{\Un(1)}\node{\Un(3){/T_3}}
\end{diagram}\] Obviously, \( P \) has codimension one in \( \Un(3)/T_2 \), and, by Definition~\ref{def:G-moment}(i), it is transverse to the fibres of the circle bundle \( \Un(3)/T_2\to \Un(3)/{T_3} \). Since \( P \) is compact, the composition \( P\to\Un(3)/{T_3} \) is a covering map. However, \( \Un(3)/{T_3} \) is simply connected, so this mapping is actually a diffeomorphism. In other words, the circle bundle  \( \Un(3)/T_2\to \Un(3)/{T_3} \) has a section. This is a contradiction because the first Chern class is non-zero.

\paragraph{Existence}
Throughout, we shall assume \( (M,g,\xi,\eta,\vp) \) is an \SST manifold on which a compact Lie group \( G \) acts freely and preserving the structure. We look for sufficient conditions to ensure the existence of a global \( G \)-moment map.

First, we recall (see \cite[Proposition 3.5.1]{ElKacimi}) the following:

\begin{theorem}[Kacimi-Alaoui]
If \( M \) is a compact, transversely K\"ahler manifold, and \( \omega \) is a basic form of type \( (1,1) \), with \( [\omega]=0 \) in basic cohomology, then there exists a basic function \( f \) such that
\( \omega=\partial\overline{\partial} f \).
\end{theorem}

An almost contact manifold is said to satisfy the \( \partial\overline{\partial} \) lemma provided the condition of this theorem holds. 

The other conditions we need are conveniently expressed in terms of equivariant cohomology.  Fix  a basis \( \{\Lel E_a\} \) of \( \g \), and denote by \( \{x^a\} \) the dual basis of \( \g^* \). Let \( \iota_a \) be the interior product with \( \Lel E_a^* \), and \( \Ld_a \) the Lie derivative with respect to \( \Lel E_a^* \). 
The Cartan model for equivariant cohomology is given by the following complex
\[C_G(\Omega(M))=\left(S(\g^*)\otimes \Omega(M)\right)^G, \quad d_G=1\otimes d-\sum_a x^a\otimes \iota_a.\]
As \( G \) acts freely on \( M \), by assumption,  equivariant cohomology reduces to cohomology on the quotient:
\[H^*(C_G(\Omega(M)),d_G)\cong H^*_{dR}(M/G)\]
(see \cite[Theorem 5.2.1]{GuilleminSternberg}).

We also need an equivariant version of the \( \partial\overline{\partial} \) lemma.  In order to state the relevant condition, it will be convenient to indicate by \( [\alpha]_0 \) the image of an equivariant form \( \alpha \) under the projection
\[\left(S(\g^*)\otimes \Omega(M)\right)^G\twoheadrightarrow \left(S(\g^*)\otimes \Omega^0(M)\right)^G.\]
\begin{definition}
An almost contact metric manifold  \( (M,g,\xi,\eta,\vp) \) endowed with the free action of a compact Lie group \( G \) \emph{satisfies the \( G \)-equivariant \( \partial\overline{\partial} \) lemma} if the following holds: whenever \( p\in S^2(\g^*)^G \) is \( d_G \)-exact then there exists \( \sigma \) such that \( [d_Gd^\vp\sigma]_0=p \).
\end{definition}

There is a purely topological condition implying the equivariant \( \partial\overline{\partial} \) lemma: the Chern-Weyl homomorphism \( \kappa_G \) fits into the commutative  diagram
\[ \begin{diagram}
    \node{(S^2(\lie{g}^*))^G} \arrow{e}\arrow{se,b} {\kappa_G} \node {H^4(C_G(M),d_G)} \arrow{s,r} {\cong} \\
\node{} \node { H^4(M/G)} 
   \end{diagram}, \]
so by assuming \( \kappa_G \) is injective  on \( (S^2(\g^*))^G \), all \( d_G \) exact forms in \( S^2(\g^*)^G \) are zero. Consequently, the equivariant \( \partial\overline{\partial} \) lemma holds trivially in this case.

In analogy with \cite{Grantcharov:Reduction}, we obtain the following existence result:

\begin{theorem}
\label{thm:reduction}
Suppose \( (M,g,\xi,\eta,\vp) \) is an \SST manifold on which a compact Lie group \( G \) acts freely preserving the structure. If the \( \partial\overline{\partial} \) lemma and the equivariant \( \partial\overline{\partial} \) lemma are satisfied, and the torsion form \( c \) extends to a closed equivariant \( 3 \)-form in \( C_G(\Omega(M)) \), then there is a global \( G \)\nobreakdash-moment map on \( M \). 

In particular, every \( G \)-invariant element of \( \g^* \) gives rise to a reduction \( M\sslash G \) with an induced \ST structure.
\end{theorem}
\begin{proof}
For an element \( \alpha\in \g^*\otimes\Omega^k(M) \), we will write \( \alpha_a \) for the contraction with \( \Lel E_a \), i.e., \( \alpha=\sum_a x^a\otimes \alpha_a \). By hypothesis, there exists \( \alpha\in\left(\g^*\otimes\Omega^1(M)\right)^G \) such that
\[0=d_G(c+\alpha)=-\sum_a x^a\otimes \iota_a c + x^a\otimes d\alpha_a -\sum_{a,b} x^ax^b\otimes \iota_b \alpha_a, \]
or phrased differently, we have
\[d\alpha_a=\iota_a c\quad \textrm{and} \quad \iota_a \alpha_b + \iota_b \alpha_a=0.\]
Now define a \( 1 \)-form \( \beta\in\lie{g}^*\otimes \Omega^1(M) \) by
\[\beta= \sum_a x^a  \otimes(\Lel E_a^*)^\flat+\alpha.\]
A standard computation, using the canonical connection \( \NB \), shows that that each component \( d\beta_a \) is both basic with respect to the Reeb foliation and of type \( (1,1) \), i.e. \( d^\vp \beta_a=d\beta_a \). By the \( \partial\overline{\partial} \) lemma, we can write \( d\beta=dd^\vp\mu \),
where \( \mu\in\g^*\otimes\Omega^0(M) \) is basic with respect to the Reeb foliation. We can assume, by averaging over \( G \), that \( \mu \) is also \( G \)-invariant. By construction,
\( \gamma=\beta-d^\vp\mu \) is \( d \) closed as well as invariant. Therefore \( d_G\gamma=\sum_{a,b} x^ax^b\iota_a \gamma_b \)
is both \( G \)-invariant and, by  \( d_G^2=0 \), \( d \)-closed. In other words, \( d_G\gamma \) equals an invariant polynomial in \( S^2(\g^*) \). 

Applying the equivariant \( \partial\overline{\partial} \) lemma to \( p=d_G\gamma \), we find \( \sigma \) such that 
\[ [d_Gd^\vp \sigma]_0=d_G\gamma. \]
With no loss of generality we can assume that \( \sigma \) is in \( \g^*\otimes \Omega^0(M) \). Setting  \( \wt\mu =\mu+ \sigma \), we compute
\[[d_Gd^\vp\wt\mu]_0=[d_G(\beta-\gamma+d^\vp\sigma)]_0=[d_G\beta]_0=\sum_{a,b}x^ax^b\iota_a\beta_b= \sum_{a,b}x^ax^b g(\Lel E_a^*,\Lel E_b^*).\]
In other terms,
\[d\wt\mu_a(\vp (\Lel E_b^*))+ d\wt\mu_b(\vp (\Lel E_a^*))\]
is a non-degenerate symmetric matrix in \( a,b \) at each point, so  \( d\wt\mu(\vp\Lel E_b^*) \) is nowhere zero. Identifying \( \wt\mu \) with an equivariant map \( M\to\g \) and composing with an isomorphism \( \g\cong\g^* \), we obtain a \( G \)-moment map, as required.
\end{proof}

\begin{remark}
\label{remark:noglobalmomentmap}
The equivariant \( \partial\overline{\partial} \) lemma is a non-trivial condition. Consider, for instance, the product of an \SST manifold \( Z \) with any K\"ahler manifold. This has a product almost contact metric structure which is in fact \SST. More concretely, take \( M=Z\times S^1\times \Un(1) \), where \( \Un(1) \) is thought of as the fibre of a principal bundle \( M\to Z\times S^1 \). The standard K\"ahler structure on the two-torus \( S^1\times \Un(1) \) induces a product \SST structure on \( M \). By construction, the torsion is basic with respect to \( \Un(1) \), and so is \( d_{\Un(1)} \) closed. On the other hand, the \( d_{\Un(1)} \) exact polynomial 
\[x^2\in \bR[x]\otimes\Omega(M)^{\Un(1)}=C_{\Un(1)}(M)\]
cannot be written as  \( [d_{\Un(1)}d^\phi\sigma]_0 \). Indeed, given a map \( \sigma\colon M\to\lie{u}(1)^*\cong\bR \), consider the embedding
\[\iota\colon S^1\to Z\times S^1\times \Un(1), \quad \iota(e^{i\theta})=(p,e^{i\theta},q), \]
and denote by \( \Lel X \) the standard generator of \( \lie{u}(1) \). Then the condition
\[0=\int_{S^1} \iota^*d\sigma =\int_0^{2\pi} d\sigma_{\iota(e^{i\theta})} (\vp (\Lel X^*)) d\theta\]
shows that \( \iota_{(\Lel X^*)}d^\vp\sigma \) cannot be identically non-zero.

The same argument shows it is not possible to find a global  \( \Un(1) \)-moment map. However, the invariant function \(\nu(p,e^{i\theta},q)=\cos\theta\)
defines a (non-global) \( \Un(1) \)-moment map, with \( M\sslash S^1\cong Z \). 
\end{remark}

We conclude the study of reductions providing a number of examples.

\begin{example}
As a first exercise, let us consider Euclidean space,\[ N=\bR\times\bC^{k+1}=\bR\times \bR^{2(k+1)},\quad g=ds^2+\sum_{j=1}^{k+1}dx_j^2+dy_j^2, \] 
with the obvious associated \ST structure: \( \xi=\partial\slash{\partial s}, \eta=ds \), and \( \vp \) induced via the standard complex structure on \( \bC^{k+1} \). This \ST manifold admits a structure preserving circle action given by \( e^{it}\cdot z:= (e^{it}z_1,\ldots,e^{it}z_{k+1}) \). Let us now introduce the \( S^1 \) invariant function \( \mu\colon\,N\to\bR \), \[ \mu(z)=-\sum_{j=1}^{k+1}|z_j|^2+1. \] Its derivative \( d\mu=-2\sum_{j=1}^{k+1}\left(x_jdx_j+y_jdy_j\right) \) has the properties \( d\mu(\xi)=0 \) and \( d\mu(\vp(X))\neq0 \); the latter follows as \( \vp \) applied to the fundamental vector field is given by \[ \vp(X)=-\sum_{j=1}^{k+1}(x_j\partial\slash{\partial x_j}+y_j\partial\slash{\partial y_j}), \] which implies \( d\mu(\vp(X))=2\sum_{j=1}^{k+1}|z_j|^2 \). Consequently, \( \mu \) is an \( S^1 \)-moment map.

The reduced space \( N\sslash{S^1} \) is \( \bR\times\bC P(k) \), and the associated \ST structure has fundamental \( 2 \)-form which can be identified with the Fubini-Study form on complex projective \( k \)-space. In this case both \( N \) and \( N\sslash{S^1} \) have null torsion; this is consistent with Proposition~\ref{prop:reduction}, since \( \alpha \) is zero, i.e. \( X \) is orthogonal to the distribution \( \Hor=\ker d\mu\cap \ker d^\vp\mu \).

Finally, note that any function \( f \) on \( \bC^{k+1} \) induces a transversal conformal transformation of the above \ST structure on \( N \). If \( f \) is chosen to be \( S^1 \) invariant, then the transformed \ST structure can be reduced via \( \mu \) to give an \ST structure on \( \bR\times\bC P(k) \).   
\end{example}

More generally, any example  of \KT reduction gives rise to an example of \ST reduction by taking an appropriate circle bundle (see Proposition~\ref{prop:sandwich} and Remark~\ref{remark:STKTreduction}). In order to produce an example on which both the manifold \( M \) and the quotient are irregular, we will make use of the following elementary observation:

\begin{proposition}
\label{prop:torusbundle}
Let \( M \) be an normal almost contact metric manifold \( (g,\xi,\eta,\vp) \) with two closed, integral \( \vp \) invariant \( 2 \)-forms \( \sigma_1 \), \(\sigma_2 \). Then the torus bundle \mbox{\( \hat M\xrightarrow{S^1\times S^1} M \)} with Chern classes \( \sigma_1,\sigma_2 \) has a family of \ST structures \( (\hat g, \hat \xi, \hat\eta,\hat\vp) \) defined as follows. Let \( \omega=(\omega_1,\omega_2) \) be a connection form with
\( \frac{1}{2\pi}d\omega_i=\pi^*\sigma_i \), and let \( X_1 \), \( X_2 \) be  the fundamental vector fields. Then 
\[\hat g(X,Y)=\omega_1(X)\omega_1(Y)+\omega_2(X)\omega_2(Y)+ g(\pi_*X,\pi_*Y),\]
and, given real constants \( s,t \),  \( s^2+t^2=1 \),
\[\hat\xi = t X_1+sX_2, \quad \hat\eta = t\omega_1+s\omega_2, \quad \hat\phi=\phi+\hat\eta^\perp\otimes \xi-\eta\otimes \hat\xi^\perp,\]
where tensors on \( M \) are lifted to \( \hat M \) horizontally using the metric, and 
\[\hat\xi^\perp = -s X_1+tX_2, \quad\hat\eta^\perp = -s\omega_1+t\omega_2.\]
\end{proposition}
\begin{proof}
By construction, \( \hat\xi \) is a Killing field, so it suffices to prove that the Sasaki-Hatakeyama \( \hat S \) vanishes, i.e.
\[ \pi_*\circ\hat S=0 = \hat\eta\circ \hat S = \hat\eta^\perp \circ \hat S. \]
 This is a straightforward computation where one uses normality and the fact that \( \sigma_1 \), \( \sigma_2 \) are of type (1,1).
\end{proof}
\begin{example}
\label{ex:threespheres}
By Proposition~\ref{prop:torusbundle}, the product of three odd-dimensional spheres has a family of \ST structures compatible with the standard metric, associated to the torus fibration
\[S^{2k+1}\times S^{2n+1}\times S^{2m+1}\to S^{2k+1}\times \CP^n\times\CP^m.\]
If we choose \( s,t \) so that \( s\slash t \) is an irrational number, then the structure is irregular. 
Consider such a structure on \( \hat M=S^1\times  S^{2n+1}\times S^{2m+1} \). A straightforward computation yields
\begin{gather*}
\hat F=\pi\sigma_1+\pi\sigma_2 + (-s\omega_1+t\omega_2)\wedge \eta,\\
\hat c=2\pi(t\sigma_1+s\sigma_2)\wedge\hat\eta +2\pi(-s\sigma_1+t\sigma_2)\wedge (-s\omega_1+t\omega_2), \quad d\hat c=4\pi^2(\sigma_1^2+\sigma_2^2).
\end{gather*}
Writing the generic element of \( \hat M \) as
\[p=(x,y_1,\dots, y_{n+1},z_1,\dots, z_{m},\rho e^{i\theta})\in \bC\times \bC^{n+1}\times \bC^{m+1},\]
we define a \( \Un(1) \) action on \( \hat M \)  by
\[e^{i\gamma}\cdot p=\bigl(x,e^{i\gamma}y_1,\dotsc, e^{i\gamma}y_{n+1}, z_1, \dotsc, z_{m},\rho e^{i(\gamma+\theta)}\bigr).\]
Thus, the fundamental vector field has the form \( X=X_1+\frac{\partial}{\partial \theta} \). If we set
\[ \mu(x,y,z)=\re x, \]
we see that \( d^\vp\mu \) is a multiple of \( \hat\eta^\perp \). Noting next
\[f=\hat\eta^\perp(X)=-s+t\rho^2,\]
we conclude that  \( \mu \) is a \( \Un(1) \)-moment map if \( \abs{s}>\abs{t} \).  By Proposition~\ref{prop:reduction}, the reduction 
\[\hat M\sslash \Un(1)=\frac{S^{2n+1}\times S^{2m+1}}{\Un(1)}.\] 
is an \ST manifold. Observe this structure is irregular, because the fundamental vector field is transverse to the two-torus generated by \( \hat \xi \). Notice also that \( X \) is not orthogonal onto  the distribution \( \mathcal{H}=\ker\hat\eta^\perp \). More precisely, projection on \( \Hor \) is given by
\[Y^\Hor = Y-f^{-1}\hat\eta^\perp(Y)X, \quad \beta^\Hor=\beta -f^{-1}\hat\eta^\perp\wedge (X\hook\beta),\]
so the \( 1 \)-form \( \alpha \) of Proposition~\ref{prop:reduction} is given by
\[\alpha=\omega_1+\rho^2d\theta -f^{-1}(1+\rho^2)\hat\eta^\perp.\]
The curvature of the circle bundle is 
\begin{gather*}
R=d(f^{-1}\hat\eta^\perp)=2\pi f^{-1}(-s\sigma_1+t\sigma_2)-2tf^{-2}\rho d\rho\wedge\hat\eta^\perp,
\end{gather*}
and therefore,  by Proposition~\ref{prop:reduction}, the torsion \( c \) of \( \hat M\sslash S^1 \) is determined by
\begin{multline*}
 \pi^*c=\hat c^\Hor -\langle \alpha,R\rangle
=2 \pi f^{-2}\rho^2(sfd\theta +t^2(\rho^2\omega_1  - \omega_2)-s(-s\omega_1+t\omega_2))\wedge \sigma_1\\
-2 \pi f^{-2}(tf\rho^2d\theta+\rho^2(st-st\rho^2+1)\omega_1 +(\rho^2(st-t^2)-1)\omega_2))\wedge \sigma_2\\
+2f^{-2}(-1+st\rho^2)\rho d\rho\wedge\omega_1\wedge\omega_2-2tf^{-2}\rho^3d\rho \wedge d\theta \wedge (-s\omega_1+t\omega_2).
\end{multline*}
Similarly, we can apply Proposition~\ref{prop:reduction} to prove that the reduction is not \SST: since \( R\hook\hat c \) is horizontal, the formula reduces to 
\begin{multline*}
\pi^*d c = (d\hat c)^\Hor-R\hook \hat c - \langle (d\alpha)^\Hor,R\rangle\\
=4 \pi^2 f^{-2}((\rho^4 t^2 + s^2\rho^2)\sigma_1^2 -\rho^2(1+2st-st\rho^2+t^2)\sigma_1\wedge\sigma_2 -(\rho^2(st-t^2)-1)\sigma_2^2)\\
-4\pi f^{-3}\rho d\rho \wedge (-s\omega_1+t\omega_2)\wedge\bigl((s+2t+t\rho^2)(-s\sigma_1+t\sigma_2) -2tf\sigma_1 +(2f^2-tf\rho^2)\sigma_2\bigr) \\
-4\pi f^{-1}\rho d\rho\wedge (d\theta-\omega_2)\wedge (-s\sigma_1+t\sigma_2) .
\end{multline*}
\end{example}
\begin{remark}
The same argument carries through replacing \( \CP^n\times\CP^m \) with an arbitrary Hermitian manifold with two closed integral \( (1,1) \)-forms \( \sigma_1,\sigma_2\). 
\end{remark}

\begin{remark}
The \( \Un(1) \)-moment map \( \mu \) of  Example~\ref{ex:threespheres} can be made global by replacing the \( S^1 \) factor in \( \hat M \) with \( \bR \).
\end{remark}

\begin{example}
\label{ex:spherereduction}
We now turn to study \ST reductions of the Sasaki spheres described in Example \ref{ex:Sasak-sphere}. On \( S^{2k+1}\subset\bC^{k+1} \), \( k>0 \), we consider the circle action given via \[ e^{is}\cdot z:=(e^{-is}z_1,e^{is}z_2,\ldots,e^{-is}z_k,e^{is}z_{k+1}).\] Clearly, the associated vector field is \[ X=\sum_{j=1}^{k+1}(-1)^j\left(x_j\partial\slash{\partial y_j}-y_j\partial\slash{\partial x_j}\right). \] Motivated by the Remark \ref{rem:moment-map-cand}, we now look at the invariant function 
\[ \mu\colon  S^{2k+1}\to\bR,\quad  \mu(z)=\mu(x,y)=\eta_z(X)=\sum_{j=1}^{k+1}(-1)^j|z_j|^2,\]
 which clearly satisfies \( d\mu(\xi)=0 \). Additionally, we see that \( d\mu(\vp(X)) \) cannot vanish on the level set \[ \mu^{-1}(0)=\left\{z\in S^{2k+1}\colon\,\sum_{j=1}^{[k/2+1]}  |z_{2j-1}|^2=\sum_{j=1}^{[(k+1)/2]} |z_{2j}|^2 \right\}, \]  which is the product of two \( k \)-dimensional spheres if \( k \) is odd, and a \( (k+1) \)-sphere with  a \( (k-1) \)-sphere if \( k \) is even. In conclusion, \( \mu \) is an \( S^1 \)-moment map.

It is clear  that the \ST reduction \( S^{2k+1}\sslash{S^1} \) can be identified with \( (S^k\times S^k)\slash{S^1} \) or  \( (S^{k+1}\times S^{k-1})\slash{S^1} \), depending on the parity of \( k \). Also note that if we apply a transversal conformal transformation, using a horizontal function of the type considered in Example \ref{ex:Sasak-sphere}, then the transformed \ST structure on \( S^{2k+1} \) can be reduced using \( \mu \); this follows by the \( S^1 \) invariance of the functions \( \sum_j\lambda_j|z_j|^2 \). 
\end{example}

\paragraph{Acknowledgements}
TBM gratefully acknowledges financial support from the \textsc{Danish Council for Independent Research, Natural Sciences}. We appreciate the useful comments from the referee.


\newpage

\begin{thebibliography}{35}
\bibitem{Baird-W:harmonic} P.~Baird, J.~C.~Wood, {Harmonic morphisms between Riemannian manifolds. London Mathematical Society Monographs. New Series, 29. The Clarendon Press, Oxford University Press, Oxford, 2003. xvi+520 pp. ISBN: 0-19-850362-8}
\bibitem{Bazzoni-O:coK} G.~Bazzoni, J.~Oprea, {On the structure of co-K\"ahler manifolds. arXiv:1209.3373 [math.DG].}
\bibitem{Blair:quasi-Sasaki} D.~E.~Blair, {The theory of quasi-sasakian structures. J. Differential Geometry, 1 (1967), 331--345.}
\bibitem{Boothby-W:contact} W.~Boothby, H.~Wang, {On contact manifolds. Ann. of Math. (2) \textbf{68} 1958 721--734.}
\bibitem{Boyer-G:Sasak-book} C.~Boyer, K.~Galicki, {Sasakian geomety. Oxford Mathematical Monographs. OUP, Oxford (2008), ISBN: 978-0-19-856495-9.}
\bibitem{Chinea-G:acm-clsI} D.~Chinea, C.~Gonzalez, {A classification of almost contact metric manifolds. Ann. Math. Pura Appl. (4) 156 (1990), 15--36.}
\bibitem{Chinea-M:acm-clsII} D.~Chinea, J.~C.~Marrero, {Classification of almost contact metric structures. Rev. Roumaine Math. Pures Appl. 37 (1992), no.~3,199--211.}
\bibitem{Chinea-al:submer} D.~Chinea~et~al, {Almost contact submersions with total space a locally conformal cosymplectic manifold. Ann. Fac. Sci. Toulouse Math. (6) 4 (1995), no.~3, 473--517. }
\bibitem{Chinea:submer} D.~Chinea, {Almost contact metric submersions. Rend. Circ. Mat. Palermo (2) 34 (1985), no.~1, 89--104.}
\bibitem{Chinea:harmonic} D.~Chinea, {On horizontally conformal \( (\vp,\vp') \)-holomorphic submersions. Houston J. Math. 34 (2008), no. 3, 721--737.} 
\bibitem{Dotti-F:ahcplx} I.~Dotti, A.~Fino, {Abelian hypercomplex 8-dimensional nilmanifolds. Ann. Global Anal. Geom. 18 (2000), no.~1, 47--59.}
\bibitem{ElKacimi} A.~El~Kacimi-Alaoui, {Opérateurs transversalement elliptiques sur un feuilletage riemannien et applications.  Compositio Math. 73 (1990), no. 1, 57–106.} 
\bibitem{Fernandez-al:SKT} M.~Fern{\'a}ndez~et-al, {Strong K\"ahler with torsion structures from almost contact manifolds. Pacific Journal of Mathematics, vol. 249, no. 1 (2011) 49--75.}
\bibitem{Fernandez-al:HKT-amc} M.~Fern{\'a}ndez~et-al, {HKT structures from almost contact manifolds. AIP Conf. Proc. 1360, 27--38.}
\bibitem{Fernandez-al:heterotic} M.~Fern{\'a}ndez~et-al, {Compact supersymmetric solutions of the heterotic equations of motion in dimension 5. Nuclear Phys. B 820 (2009), no. 1-2, 483--502.} 
\bibitem{Fino-al:SKT} A.~Fino, M.~Parton, S.~Salamon, {Families of strong KT structures in six dimensions. Comment. Math. Helv. 79 (2004), no.~2, 317--340.}
\bibitem{Friedrich-I:spinors} T.~Friedrich, S.~Ivanov, {Parallel spinors and connections with skew-symmetric torsion in string theory. 
Asian J. Math. 6 (2002), no. 2, 303--335.} 
\bibitem{Friedrich-I:contact} T.~Friedrich, S.~Ivanov, {Almost contact manifolds, connections with torsion, and parallel spinors. 
J. Reine Angew. Math. 559 (2003), 217--236.} 
\bibitem{Gauduchon:Dirac} P.~Gauduchon, {Hermitian connections and Dirac operators. 
Boll. Un. Mat. Ital. B (7) 11 (1997), no. 2, suppl., 257--288.}
\bibitem{Grantcharov-O:Sasak-red} G.~Grantcharov, L.~Ornea, {Reduction of Sasakian manifolds. J. Math. Phys. 42 (2001), no.~8, 3809--3816.}
\bibitem{Grantcharov:Reduction} Grantcharov, Gueo; Papadopoulos, George; Poon, Yat Sun. {Reduction of HKT-structures. J. Math. Phys. 43 (2002), no. 7, 3766--3782.} 
\bibitem{Grantcharov-G-P:CYT} G.~Grantcharov~et~al, {Calabi-Yau connections with torsion on toric bundles}, {J.~Differential~Geom. 78 (2008), no.~1, 13--32.}
\bibitem{GuilleminSternberg} V.~Guillemin, S.~Sternberg, {Supersymmetry and equivariant de Rham theory. Springer, 1999, ISBN: 3-540-64797-X.}
\bibitem{Houri-al:ST} T.~Houri, H.~Takeuchi, Y.~Yasui, {A Deformation of Sasakian Structure in the Presence of Torsion and Supergravity Solutions. arXiv:1207.0247v1 [hep-th].}
\bibitem{Ivanov-P:vanishing-string} S.~Ivanov, G.~Papadopoulos, {Vanishing theorems and string backgrounds. Classical Quantum Gravity 18 (2001), no.~6, 1089--1110.}
\bibitem{Joyce:hclpx-quat-quo} D.~Joyce, {The hypercomplex quotient and the quaternionic quotient. Math. Ann. 290, 323--340 (1991).}
\bibitem{Marrero:trans-Sasak} J.~C.~Marrero, {The Local Structure of Trans-Sasakian Manifolds. Ann. Mat. Pura Appl. (4) 162 (1992), 77--86.}
\bibitem{Morimoto:acm-normal} A.~Morimoto, {On normal almost contact structures. J. Math. Scc. Japan, vol.~15, no.~4 (1963) 420--436.}  
\bibitem{Puhle:qSasak} C.~Puhle, {Almost contact metric 5-manifolds and connections with torsion. Differential Geom. Appl.  30  (2012),  no. 1, 85--106.}
\bibitem{Salamon:tour} S.~Salamon, {A tour of exceptional geometry. Milan J. Math. 71 (2003), 59--94.}
\bibitem{Sasaki-H:Nijenhuis} S.~Sasaki, Y~Hatakeyama, {On differentiable manifolds with certain structures which are closely related to almost contact structures, II. Tohoku Math. J. 13 (1961) 281--294.}
\bibitem{Spindel-al} Ph.~Spindel~et~al, {Extended supersymmetric \( \sigma \)-models on group manifolds. I: The complex structures. Nuclear Phys. B \textbf{308} (1988), 662--698.}
\bibitem{Strominger:superstrings} A.~Strominger, {Superstrings with torsion, Nucl.~Phys.~B 274 (1986), 253--284.}
\bibitem{Tanno:homoth} S.~Tanno, {The topology of contact Riemannian manifolds. Illinois J. Math. 12 (1968) 700--717.}
\bibitem{Ugarte:balanced} L.~Ugarte, {Hermitian structures on six-dimensional nilmanifolds. Transformation Groups, vol. 12, no.~1 (2007), 175--202.} 
\end{thebibliography}
\end{document}